\documentclass[a4paper,10pt,reqno]{amsart}

\usepackage{graphicx,amsmath,amssymb}
\usepackage{amssymb}
\usepackage{mathrsfs}
\usepackage{subfigure}
\usepackage{epstopdf} 
\usepackage{paralist}

\newtheorem{theorem}{Theorem}
\newtheorem{lemma}{Lemma}
\newtheorem{definition}{Definition}
\newtheorem{remark}[theorem]{Remark}

\let\pa\partial
\newcommand{\T}{\mathcal{T}}
\let\r\rho

\usepackage{cite}

\title{Lane Formation by side-stepping} 

\author{Martin Burger$^*$}

\author{Sabine Hittmeir$^\ddagger$}

\author{Helene Ranetbauer$^\ddagger$}

\author{Marie-Therese Wolfram$^\ddagger$}

\begin{document}
\maketitle
\renewcommand{\thefootnote}{\fnsymbol{footnote}}
\footnotetext[1]{Institut f\"ur Numerische und Angewandte Mathematik, Westf\"alische Wilhelms Universit\"at
M\"unster, Einsteinstrasse 62, D 48149 M\"unster, Germany (martin.burger@wwu.de).}
\footnotetext[3]{Radon Institute for Computational and Applied Mathematics, 
Austrian Academy of Sciences, Altenberger Strasse 69, 4040 Linz, Austria (sabine.hittmeir@ricam.oeaw.ac.at, helene.ranetbauer@ricam.oeaw.ac.at, mt.wolfram@ricam.oeaw.ac.at). Supported by the Austrian Academy of Sciences \"OAW via the New Frontiers project NST-0001.}
\renewcommand{\thefootnote}{\arabic{footnote}}

\pagestyle{myheadings}
\thispagestyle{plain}
\markboth{Burger, Hittmeir, Ranetbauer, Wolfram}{Lane Formation for a pedestrian model}

\begin{abstract}
In this paper we study a system of nonlinear partial differential equations, which describes the evolution
of two pedestrian groups moving in opposite direction. The pedestrian dynamics are driven by aversion and cohesion, i.e.
the tendency to follow individuals from the own group and step aside in the case of contraflow. We start with a 2D lattice
based approach, in which the transition rates reflect the described dynamics, and derive the corresponding PDE system
by formally passing to the limit in the spatial and temporal discretization. We discuss the existence of special stationary solutions,
which correspond to the formation of directional lanes and prove existence of global in time bounded weak solutions. The proof is based on
an approximation argument and entropy inequalities. Furthermore we illustrate the behavior of the system with numerical simulations.
\end{abstract}

\bigskip

\begin{small}
\end{small}


\section{Introduction}

In the last decades demographics, urbanization and changes in our society resulted in an increased emergence of large pedestrian crowds, for example the commuter traffic in urban underground stations, political demonstrations or the evacuation of large buildings.  Understanding the dynamics of these crowds has become a fast growing and important field of research. The first research activities started in the field of transportation research, physics and social sciences, but the ongoing development of mathematical models initiated a lot of research also in the applied mathematics community. Nowadays mathematical tools to analyze and investigate the derived models provide useful new insights into the dynamics of pedestrian crowds.\\ 
A variety of different mathematical models has been proposed in the past which can be generally classified into microscopic and macroscopic approaches. In the microscopic framework the dynamics of each individual is modeled taking into account social interactions with all others as well as interactions with the physical surrounding. This approach results in high dimensional and very complex systems of equations. Examples include the social force model by Helbing (cf. \cite{Chraibi2011425}, \cite{helbing1995social}, \cite{helbing2000simulating}), cellular automata (cf. \cite{kirchner2002simulation}, \cite{blue2001cellular}, \cite{fukui1999self}, \cite{anguige2009one}) or stochastic optimal control approaches, cf. \cite{hoogendoorn2004pedestrian}.\\
Macroscopic models, where the crowd is treated as a density, can be derived by coarse graining procedures from microscopic equations (see e.g. \cite{burger2011continuous}), leading to nonlinear conservation laws or coupled systems of such (see e.g.  \cite{hughes2002continuum}, \cite{colombo2005pedestrian}, \cite{colombo2012class}). Other approaches heuristically motivating macroscopic models are based upon optimal transportation theory, cf. \cite{maury2010macroscopic}, mean field games (cf. \cite{lasry2007mean}, \cite{lachapelle2011mean}, \cite{Burger20141311}) or optimal control cf. \cite{fronasier2014}. Piccoli and co-workers (cf. \cite{piccoli2009pedestrian} and \cite{cristiani2011multiscale}) proposed a measure based approach capable to describe pedestrian dynamics on both microscopic and macroscopic scale - hence bridging the gap between the two description levels. Recently there has been an increasing interest in kinetic models and their respective hydrodynamic limits in pedestrian dynamics, see for example \cite{moussaid2012traffic} and \cite{Degond2013809}.\\ 
For an extensive review on the mathematical literature concerning crowd dynamics and the closely related field of traffic dynamics we refer to \cite{bellomo2011modeling}.

\noindent In this paper  we (formally) derive and rigorously analyze a PDE system describing the evolution of two pedestrian groups moving in opposite direction. The individual dynamics are driven by two forces, cohesion and aversion. We show that this minimal dynamics already result in complex macroscopic features, namely the formation of directional lanes. 
We start with a 2D lattice model, in which the transition rates, i.e. the rate at which a particle jumps from one site to the next, express the tendency of individuals to stay within their own group (i.e. follow individuals moving in the same direction) while stepping aside when individuals from the other group approach.  The corresponding mean-field PDE model can be derived by a  Taylor expansion 
(up to second order) and is a nonlinear cross diffusion system with degenerate mobilities.\\
Similar models have been proposed in the literature, for example in the context of ion transport, cf. \cite{MR2745794} or population dynamics, cf. \cite{zamponi2015analysis}. The coherent difference of our model to these works are additional challenging features, namely a perturbed gradient flow structure as well as an anisotropic degenerate diffusion matrix. Although the system lacks the classical gradient flow structure, we can show that the entropy grows at most linearly in time. The corresponding entropy estimates are a crucial ingredient for deriving the global existence result for bounded weak solutions. The existence proof is based on an implicit time discretization and an $H^1$-regularization of the 
time-discrete problem. Note that we follow a different approach than J\"ungel in \cite{jungel2014boundedness}, which has the advantage that the method is based on an $H^1$-regularization only and does not require the additional Bilaplace operator. We define a fixed point operator in $L^2(\Omega)$ and use Schauder's fixed point theorem to deduce the existence of a solution to the regularized problem. The derived entropy estimates as well as a generalized version of the Aubin-Lions lemma justify the limit in the regularization parameter. 

\noindent This paper is organized as follows: In Section \ref{modelling} we present the 2D lattice based model and derive its (formal) mean-field limit via Taylor expansion up to second order. Furthermore we discuss the existence of special stationary solutions in Section \ref{stationary_solutions}. Section \ref{basic_properties} focuses on structural features of the resulting PDE system, such as the corresponding entropy functional and the related dissipation inequality. Furthermore we study the boundedness of the densities, which is an essential prerequisite for the global existence proof outlined in Section \ref{mainresult}. Finally we illustrate the behavior of the model with various numerical experiments, which reproduce well known phenomena such as lane formation in Section \ref{numerical_simulations}. 
 
\section{Mathematical modelling}\label{modelling}
In this section we present the formal derivation of the proposed PDE model from a microscopic discrete lattice approach. We consider two groups of individuals moving in opposite direction, i.e. one group is moving to the right, the other to the left. The individual dynamics are driven by two basic objectives: first individuals try to stay within or close to their own group, i.e. pedestrians walking in the same direction. Moreover they step aside when being approached by an individual moving in the opposite direction. Based on this minimal interaction rules we derive the corresponding
PDE model by Taylor expansion up to second order in the following.

\subsection{The microscopic model}
Throughout this paper we refer to the groups of individuals moving to the right and left as red and blue individuals respectively. Their dynamics are driven by the objectives described above and correspond to cohesion and aversion. Let us consider a domain $\Omega \subseteq \mathbb{R}^2$, partitioned into an equidistant grid of mesh size $h$. Each grid point $(x_i, y_j) = (ih, jh)$, $i\in \{0,\ldots N\}$ and $j\in \{0, \ldots M\}$ can be occupied by either a red or a blue individual.
The probability to find a red individual at time $t$ at location $(x_i,y_j)$ is given by:
\[r_{i,j}(t)= P(\text{red individual is at position } (x_i,y_j) \text{ at time t}),\]
with an analogous definition for $b_{i,j}(t)$. We set $t_k=kh$ for $k\in \mathbb{N}$ and use the abbreviation $r_{i,j}=r_{i,j}(t_k)$ if the time step $t_k$ is obvious.
The dynamics of the individuals are driven by the evolution of the probabilities $r$ and $b$. These probabilities depend on the transition rates of individuals. Let $\T^{\{i,j\}\rightarrow\{i+1,j\}}$ denote the transition rate of an individual to move from the discrete point $(x_i,y_j)$ to $(x_{i+1},y_j)$.
We define the transition probabilities for the reds as:
\begin{align}\label{transprob}
\begin{aligned}
\T_r^{\{i,j\}\rightarrow \{i+1,j\}}&=(1-\rho_{i+1,j})(1+\alpha \, r_{i+2,j}),\\
\T_r^{\{i,j\}\rightarrow \{i,j-1\}}&=(1-\rho_{i,j-1})(\gamma_0+\gamma_1 \, b_{i+1,j}),\\
\T_r^{\{i,j\}\rightarrow \{i,j+1\}}&=(1-\rho_{i,j+1})(\gamma_0+\gamma_2 \, b_{i+1,j}),
\end{aligned}
\end{align}
where $\rho=r+b$, $0\leq \gamma_0,\gamma_1,\gamma_2\leq 1$ and $0\leq\alpha\leq\frac{1}{2}$. 
The factor $(1-\rho)$ corresponds to size exclusion, i.e. an individual cannot jump into the neighboring cell if it is occupied. Note that we assume that individuals only anticipate the dynamics in their direction of movement, i.e. they do not look backwards, which is reasonable when modeling the movement of pedestrians. The second factor in the transition probabilities (\ref{transprob}) corresponds to cohesion and aversion. If $\alpha >0$ the probability of moving in the walking direction is increased if the individual in front, i.e. at position $(x_{i+2},y_{j})$, is moving in the same direction (assuming that the cell $(x_{i+1},y_j)$ is not occupied). \\
Aversion corresponds to sidestepping. If $\gamma_1\geq \gamma_2>0$, an individual steps aside if another individual, in \eqref{transprob} a blue particle located at $(x_{i+1},y_j)$, is approaching. If $\gamma_1>\gamma_2$, there is a preference to make a step to the right hand side with respect to their direction of movement, if $\gamma_2>\gamma_1$, to the left. From the perspective of an observer red individuals prefer to make a jump down if a blue individual is ahead of them in the case $\gamma_1>\gamma_2$. The parameter $\gamma_0>0$ includes diffusion in the $y$-direction. In the case of no diffusion, i.e. $\gamma_0=0$, individuals only step aside when being approached by an individual moving in opposite direction.\\
The master equation for the red particles then reads as
\begin{align}\label{master_r}
r_{i,j}(t_{k+1})&=r_{i,j}(t_k)+\T_r^{\{i-1,j\}\rightarrow \{i,j\}} r_{i-1,j}(t_k)\nonumber \\
&\quad+ \T_r^{\{i,j+1\}\rightarrow \{i,j\}} r_{i,j+1}(t_k)+\T_r^{\{i,j-1\}\rightarrow \{i,j\}} r_{i,j-1}(t_k)\\
&\quad  - \left(\T_r^{\{i,j\}\rightarrow \{i+1,j\}}+\T_r^{\{i,j\}\rightarrow \{i,j-1\}}+\T_r^{\{i,j\}\rightarrow \{i,j+1\}}\right) r_{i,j}(t_k). \nonumber
\end{align}
The probability to find a red particle at location $(x_i, y_j)$ in space corresponds to the probability that a particle located at $(x_{i-1},y_j)$ jumps forward (first term), particles
located above or below, i.e. at $(x_i,y_{j \pm 1})$ jump up or down (second line), minus the probability that a particle located at $(x_i,y_j)$ moves forward or steps
aside (third line). The corresponding transition rates for the blue particles are defined accordingly to (\ref{transprob}) by:
\begin{align}
\begin{aligned}
\T_b^{\{i,j\}\rightarrow \{i-1,j\}}&=(1-\r_{i-1,j})(1+\alpha \, b_{i-2,j}),\\
\T_b^{\{i,j\}\rightarrow \{i,j+1\}}&=(1-\r_{i,j+1})(\gamma_0 + \gamma_1 \, r_{i-1,j}),\\
\T_b^{\{i,j\}\rightarrow \{i,j-1\}}&=(1-\r_{i,j-1})(\gamma_0+\gamma_2 \, r_{i-1,j}).
\end{aligned}
\end{align}
The master equation for the blue particles has the same structure as \eqref{master_r}, i.e.: 
\begin{align}\label{master_b}
b_{i,j}(t_{k+1})&=b_{i,j}(t_k)+\T_b^{\{i+1,j\}\rightarrow \{i,j\}} b_{i+1,j}(t_k)\nonumber\\
&\quad+ \T_b^{\{i,j-1\}\rightarrow \{i,j\}} b_{i,j-1}(t_k) + \T_b^{\{i,j+1\}\rightarrow \{i,j\}} b_{i,j+1}(t_k)\\
&\quad - \left(\T_b^{\{i,j\}\rightarrow \{i-1,j\}}+\T_b^{\{i,j\}\rightarrow \{i,j+1\}}+\T_b^{\{i,j\}\rightarrow \{i,j-1\}}\right) b_{i,j}(t_k). \nonumber
\end{align}

\subsection{Derivation of the macroscopic model}
In the following we shall consider the formal limit $h=\Delta t=\Delta x=\Delta y \rightarrow 0$ in equations (\ref{master_r}) and (\ref{master_b}) to derive the 
corresponding PDE system. After performing a Taylor expansion up to second order, we obtain
\begin{align}\label{system}
\begin{aligned}
\pa_t r &=-\nabla \cdot J_r,\\
\pa_t b &=-\nabla \cdot J_b,
\end{aligned}
\end{align}
where 
\[J_r:=\begin{pmatrix}
(1-\r)(1+\alpha r)r+\frac{h}{2}\left[\pa_x(r(1-\r)(1+\alpha r))-2((1-\r)\pa_x r)\right]\\\\
-(\gamma_1-\gamma_2)(1-\rho)b r-\frac{h}2\left[(\gamma_1+\gamma_2)\left((1-\r)\pa_y(rb)+br\pa_y\r\right) \right.\\
\left.+2\gamma_0\left((1-\rho)\pa_y r +r\pa_y\rho\right)+2(\gamma_1-\gamma_2)(1-\r)r\pa_xb  \right]\\
\end{pmatrix},\]
and
\[J_b:=\begin{pmatrix}
-(1-\r)(1+\alpha b)b+\frac{h}{2}\left[\pa_x(b(1-\r)(1+\alpha b))-2((1-\r)\pa_x b)\right]\\\\
(\gamma_1-\gamma_2)(1-\rho)b r-\frac{h}2\left[(\gamma_1+\gamma_2)\left((1-\r)\pa_y(rb)+br\pa_y\r\right)\right.\\
\left. +2\gamma_0\left((1-\rho)\pa_y b +b\pa_y\rho\right)+2(\gamma_1-\gamma_2)(1-\r)b\pa_x r  \right]\\
\end{pmatrix},\]
denote the fluxes for $r$ and $b$ respectively. The first order terms correspond to the movement of the reds and blues to the right and left in $x$-direction respectively as well as to the preference of either stepping to the right or left in $y$-direction (depending on the difference $\gamma_1 - \gamma_2$). The second order terms correspond to the cross diffusion terms where the prefactor $h$ denotes the lattice size.
We consider system (\ref{system}) on $\Omega\times (0,T)$, where $\Omega \subseteq \mathbb{R}^2$ is a bounded domain.
In our computational examples, see Section \ref{numerical_simulations}, the domain $\Omega$ corresponds to a corridor, i.e. $\Omega=[-L_x,L_x]\times[-L_y,L_y]$ with $L_y \ll L_x$. As individuals cannot penetrate the walls, we set no flux boundary conditions on the top and bottom, i.e.

\begin{align*}
J_{r,b}\cdot \begin{pmatrix}
0\\\pm 1
\end{pmatrix} =0  \qquad \textnormal{at} \ y=\pm L_y\,.
\end{align*}
At the entrance and exit of the corridor, i.e. at $x=\pm L_x$, we assume periodic boundary conditions.
Note that Robin type boundary conditions, where the in- and outfluxes at the entrance and exits are directly proportional to the local density, would be more realistic. 
The boundary conditions set above correspond to the simplest choice and shall serve as a starting point for the investigation of more realistic and complex models in the near future, cf. \cite{burger2015flow}.\\
We would like to remark that the lengthy Taylor expansion and formal limiting procedure can be accomplished automatically using computer algebra techniques, even for more general classes of models, see \cite{koutschan2015symbolic}.

\subsection{Stationary Solutions} \label{stationary_solutions}

In this last part of the modelling section we study the existence of specific stationary solutions, which correspond to the formation of lanes. These segregation phenomena can be observed in crowded streets with pedestrians as well as in experiments. Lane formation is a rather intuitive phenomenon, but a strict mathematical definition is less obvious. In the following we shall distinguish between strict segregation and the case when still some pedestrians might get into the counterflow, leading to the definition:
\medskip
\begin{definition}\label{d:lanes}
Let $(r,b)$ denote a stationary solution to system \eqref{system} for $\gamma_1 > \gamma_2$, which is x-independent, i.e. for all $x,x_0\in[-L_x,L_x]$ and any $y\in [-L_y,L_y]$ we have $(r,b)(x,y)=(r,b)(x_0,y)$. Considering therefore $(r,b)$ as a function of $y$ only, we 
 call $(r,b)\in L^\infty[-L_y,L_y]\times L^\infty[-L_y,L_y]$
\medskip
\begin{itemize}
\item \emph{a solution with strong lane formation}, if the functions $r$ and $b$ have a compact support in y-direction with
\begin{align*}
\operatorname{supp}(r) \cap \operatorname{supp}(b) =\emptyset\qquad 
\textnormal{and}\quad   
\sup_{y\in[-L_y,L_y]}\{\operatorname{supp}(r)\} \leq \inf_{y \in[-L_y,L_y]}\{\operatorname{supp}(b)\}. 
\end{align*}
\item \emph{a solution with weak lane formation}, if the sufficiently smooth solution $(r,b)$ satisfies  
\begin{align*}
\partial_y r < 0 \quad \text{ and }  \quad \partial_y b > 0\,,
\end{align*}
and there exists a point $\tilde{y} \in (-L_y, L_y)$, such that 
\begin{align*}
r(\tilde y)=b(\tilde y)\,.
\end{align*}
\end{itemize}
\end{definition}
\noindent Note that the definition of weak lane formation has to be changed accordingly if individuals have the preference to step to the left instead of right, i.e. $\gamma_2>\gamma_1$.
\noindent We expect that the side-stepping initiates the formation of directional lanes, an assumption that has also been confirmed by the numerical experiments in Section \ref{numerical_simulations} for specific ranges of parameters. In particular we consider system \eqref{system} in the case $0<\alpha\leq\frac{1}{2}$ and $\gamma_1-\gamma_2=\mathcal{O}(h)$, i.e.
\begin{align}\label{equ.system2}
\begin{aligned}
\pa_tr&=-\pa_x\left((1-\rho)(1+\alpha r)r\right)+h\,\pa_y\left((1-\rho)b r\right)-\frac{h}{2}\pa_x^2(r(1-\rho)(1+\alpha r))\\
&\quad +h\pa_x((1-\rho)\pa_x r)+\frac{h}{2}(\gamma_1+\gamma_2)\pa_y\left((1-\rho)\pa_y(rb)+br\pa_y\rho\right)\\
&\quad +h\gamma_0\pa_y\left((1-\rho)\pa_y r +r\pa_y\rho\right)\\
\pa_tb &=\pa_x\left((1-\rho)(1+\alpha b)b\right)-h\,\pa_y\left((1-\rho)b r\right)-\frac{h}{2}\pa_x^2(b(1-\rho)(1+\alpha b))\\
&\quad +h\pa_x((1-\rho)\pa_x b)+\frac{h}{2}(\gamma_1+\gamma_2)\pa_y\left((1-\rho)\pa_y(rb)+br\pa_y\rho\right)\\
&\quad+h\gamma_0\pa_y\left((1-\rho)\pa_y b +b\pa_y\rho\right),
\end{aligned}
\end{align}
where we set without loss of generality $\gamma_1-\gamma_2=h$. Note that the second order terms in $h$ are dropped out, i.e. the terms including $x$- and $y$-derivatives are neglected. 
In this case we can prove \textit{weak lane formation} for $\gamma_0 > 0$ and postulate the formation of \textit{strong lanes} as $\gamma_0 \rightarrow 0$.

\noindent Therefore, we consider system \eqref{equ.system2} and analyze its equilibrium solutions which are constant in $x$-direction. In this case system \eqref{equ.system2} reduces to
\begin{subequations}
\begin{align}
0 &=(1-\rho)rb+\frac{\gamma_1+\gamma_2}{2}\left((1-\rho)\pa_y(rb)+br\pa_y\rho\right)+\gamma_0\left((1-\rho)\pa_y r +r\pa_y\rho\right),\label{equ.system4_r}\\
0 &=-(1-\rho)rb+\frac{\gamma_1+\gamma_2}{2}\left((1-\rho)\pa_y(rb)+br\pa_y\rho\right)+\gamma_0\left((1-\rho)\pa_y b +b\pa_y\rho\right).\label{equ.system4_b}
\end{align}
\label{equ.system4_rb}
\end{subequations}
Note that we have assumed a preference for stepping to the right in \eqref{equ.system2}, which corresponds to the different sign in the first terms of \eqref{equ.system4_rb}. 
If $\rho<1$, we can rewrite \eqref{equ.system4_rb} as 
\begin{subequations}
\begin{align}
0 &=\frac{rb}{1-\rho}+\frac{\gamma_1+\gamma_2}{2}\pa_y\left(\frac{rb}{1-\rho}\right)+\gamma_0\pa_y\left(\frac{r}{1-\rho}\right),\label{equ.system3_r}\\
0 &=-\frac{rb}{1-\rho}+\frac{\gamma_1+\gamma_2}{2}\pa_y\left(\frac{rb}{1-\rho}\right)+\gamma_0\pa_y\left(\frac{b}{1-\rho}\right).\label{equ.system3_b}
\end{align}
\label{equ.system3_rb}
\end{subequations}
Summation of \eqref{equ.system3_r} and \eqref{equ.system3_b} and subsequent integration gives
\begin{align}\label{equ.solutioncurve}
(\gamma_1+\gamma_2)\frac{rb}{1-\rho}+\gamma_0\frac{\rho}{1-\rho}=C,
\end{align}
for some constant $C$. Equation \eqref{equ.system3_rb} allows us to study the behavior of stationary solution curves with respect to the densities $r$ and $b$. Figure \ref{f:stat1} illustrates these stationary solutions in the case $\gamma_1-\gamma_2=h$ for different values of $C$. 
\begin{figure}[h!]
\begin{center}
\includegraphics[width=83mm]{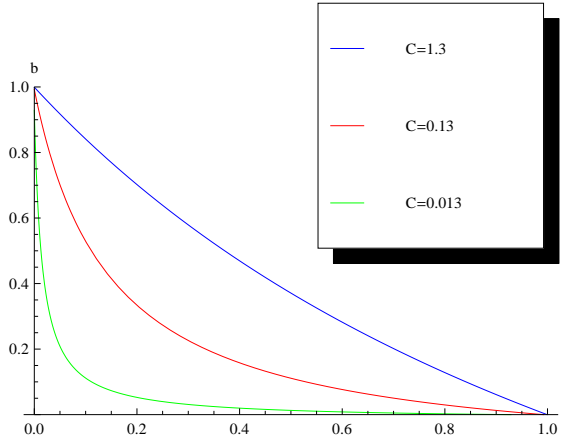}
\caption{Stationary solution curves for $\gamma_1=0.5$, $\gamma_2=0.4$ and $\gamma_0=0.001$}\label{f:stat1}
\end{center}
\end{figure}
If $r=0$ or $b=0$ then $\pa_y r=0$ or $\pa_y b=0$ respectively. Hence solution curves can get arbitrarily close to the $r$- and $b$-axes, but they can only reach them in the case of a trivial solution curve, i.e. consisting only of one stationary point lying on one of the axes. The actual starting and end points of the solution curves as well as the corresponding constants $C$ depend the chosen parameters and on the initial masses of the system, i.e. on
\[ M_r := \int_\Omega r \,dx\,dy \text{ and } M_b := \int_\Omega b\,dx\,dy.\]  
In the case of small values of $C$ we observe a quick change of the densities $r$ and $b$ from high to low values and the other way around. For larger values the densities increase or respectively decrease
slower along the solution curves.\\
The following additional solution properties can be deduced from equations  \eqref{equ.system3_rb} and \eqref{equ.solutioncurve}.
\medskip
\begin{lemma}\label{l:statsol}
Let $(r,b)$ denote solutions to system \eqref{equ.system3_rb} and let $\tilde{C}\in \mathbb{R}^+$ be a constant with $0<\tilde{C}<1$. 
\begin{compactenum}[(i)]
\item There exists no solution $(r,b)$ with $\rho \equiv \tilde{C}$ and $r,b>0$.\label{5}
\item There exists no solution $(r,b)$ with $r\equiv b>0$. \label{6}
\item Any solution $(r,b)$ is monotone  with $\pa_y r<0$ and $\pa_y b>0$.\label{7}
\end{compactenum}
\end{lemma}
\begin{proof}
To show \eqref{5} we assume to the contrary that there exists a solution with $\rho \equiv \tilde{C}$ and $r,b>0$. Then equation \eqref{equ.solutioncurve} implies that $rb$ is a positive constant and therefore the same holds true for $r$ and $b$ individually. This is a contradiction to \eqref{equ.system3_rb} as $\frac{rb}{1-\rho}=-\frac{rb}{1-\rho}$ is only true if $rb=0$.

To see \eqref{6} we again argue by contradiction. If $r\equiv b$ we immediately deduce from system \eqref{equ.system3_rb} that $r\equiv b\equiv 0$. 

To prove the monotinicity properties in \eqref{7}
we first observe that \eqref{equ.system3_r} and \eqref{equ.system3_b} imply
\begin{align}
&\pa_y \left( \frac{rb+C_1r}{1-\rho}\right)<0 \quad \text{ and } \quad \pa_y \left( \frac{rb+C_1b}{1-\rho}\right)>0,\label{equ.ineq1}
\end{align} 
for some constant $C_1>0$. This allows to exclude the existence of a $\bar y \in [-L_y,L_y]$ with $\pa_y r(\bar y)=0$, since in this case equations \eqref{equ.ineq1} would yield $\pa_y b(\bar y)<0$ as well as $\pa_y b(\bar y)>0$. Therefore $r$ has to be monotone and due to symmetry $b$ is also monotone with the opposite sign. \\
To show the stated signs of the derivatives we assume $\pa_y r>0$ and $\pa_y b<0$. Subtracting the equations in \eqref{equ.ineq1} then leads to $(r-b)\pa_y \rho<0$ and thus to a contradiction, since for $\pa_y \rho >0$ equation \eqref{equ.ineq1} as well as the assumptions imply $r-b>0$, whereas for $\pa_y \rho <0$ the second equation in \eqref{equ.ineq1} gives $r-b<0$. We therefore obtain the desired monotonicity properties $\pa_y r<0$ and $\pa_y b>0$.\qquad
\end{proof}
\smallskip
Lemma \ref{l:statsol} indicates the existence of weak lane formation. From \eqref{7} we know that $r$ and $b$ are monotone functions which are strictly positive. Hence there
exists a single point $\tilde y \in (-L_y, L_y)$ where $r(\tilde{y}) = b(\tilde{y})$. Due to the side-stepping tendency the reds will move to the bottom, while the blues move
up. In the case of equal masses it is impossible that one density is larger than the other on the whole domain, which implies the formation of weak lanes in the sense of Definition \ref{d:lanes} in this case:
\smallskip
\begin{theorem}
Let $\gamma_0 > 0,~\gamma_1 > \gamma_2$ and $M_r = M_b = M$. Then system \eqref{equ.system3_rb} has non-trivial stationary states, and any stationary solution constant in the $x$-direction exhibits weak lane formation.
\end{theorem}

Further properties of solutions to \eqref{equ.system4_rb} and \eqref{equ.system3_rb} can be observed for different asymptotic parameter regimes:
\begin{itemize}
\item $\gamma_0\to0$: We deduce from \eqref{equ.solutioncurve} and \eqref{equ.system3_rb} that $rb\to 0$, i.e \emph{the smaller $\gamma_0$, the sharper the separation of $r$ and $b$}. 
\item $\gamma_1+\gamma_2\to0$: In this case $\rho \equiv \tilde{C}$ for some constant $0<\tilde{C}<1$ and $\pa_y r=C_1r(\tilde{C}-r)$ for some constant $C_1>0$ which corresponds to lane formation.
\end{itemize} 
\smallskip
\begin{remark}
If pedestrians have the preference to step to the left instead of to the right the monotonicity behavior of $r$ and $b$ is reversed.
\end{remark}
\section{Basic properties} \label{basic_properties}
In this section we discuss basic properties of system (\ref{system}). In the following we set $\gamma_0>0$, $\gamma:=\gamma_1=\gamma_2$ and $\alpha=0$. Then system (\ref{system}) reads as
\begin{align}
\label{system_rb}
\begin{aligned}
\pa_tr&=-\pa_x\left((1-\rho)r\right) +\frac{h}{2}\pa_x\left((1-\rho)\pa_x r+r \pa_x \rho\right)\\
&\quad +h\left[\gamma_0\pa_y\left((1-\rho)\pa_y r +r\pa_y\rho\right)+\gamma\pa_y\left((1-\r)\pa_y(rb)+br\pa_y\r\right)\right]\\
\pa_tb &=\pa_x\left((1-\r)b\right)+\frac{h}{2}\pa_x\left((1-\rho)\pa_x b+b \pa_x \rho\right)\\
&\quad +h\left[\gamma_0\pa_y\left((1-\rho)\pa_y b +b\pa_y\rho\right)+\gamma\pa_y\left((1-\r)\pa_y(rb)+br\pa_y\r\right)\right].\\
\end{aligned}
\end{align}
We shall prove global existence of weak solutions of system \eqref{system_rb} in Section \ref{mainresult}, a result which can be extended to the case $\alpha>0$ and $\gamma_1-\gamma_2=\mathcal{O}(h)$. The proof uses several structural features of system \eqref{system_rb}, such as the corresponding entropy functional and the boundedness of solutions, which we discuss in this section.
\bigskip

\subsection{Entropy functional}\label{entropy_disspation}
A key point in the existence analysis are estimates based on the corresponding entropy functional
\begin{align}\label{entropy}
\begin{aligned}
\mathcal{E}:=\int_{\Omega} r(\log r-1)&+b(\log b-1)\,dx\,dy\\
&+\int_{\Omega}\frac{1}{2}(1-\rho)(\log (1-\rho)-1)+\frac{2}{h} r V_r+\frac{2}{h} bV_b \, dx \,dy,
\end{aligned}
\end{align}
where the potentials $V_r(x,y)=-x$ and $V_b(x,y)=x$ correspond to the motion of the red and blue individuals to the right and left respectively.
Note the difference in the prefactor $\frac{1}{2}$ of the entropy term $(1-\rho)(\log(1-\rho)-1)$ compared to other entropies used in the literature for similar PDE models, cf. \cite{MR2745794}, \cite{zamponi2015analysis}. This prefactor results from the anisotropic diffusion as we shall explain in the following.\\
Introducing the entropy variables $u$ and $v$  
\begin{equation*} 
u:=\pa_r \mathcal{E}=\log r -\frac{1}{2}\log(1-\rho)+\frac{2}{h} V_r \quad \text{and} \quad v:=\pa_b \mathcal{E}=\log b -\frac{1}{2}\log(1-\rho)+\frac{2}{h} V_b,
\end{equation*}
we can rewrite \eqref{system_rb} as follows
\begin{align}
\label{system_entropy}
\begin{pmatrix}
\pa_t r \\ \pa_t b
\end{pmatrix} &=\frac{h}{2}\begin{pmatrix}
\nabla & 0\\ 0 & \nabla 
\end{pmatrix} \cdot\left( M\begin{pmatrix}
\pa_x u\\ \pa_y u \\ \pa_x v \\ \pa_y v
\end{pmatrix}
+\begin{pmatrix}
\frac{r}{2}\pa_x \rho \\ \gamma_0 r \pa_y \rho \\ \frac{b}{2}\pa_x \rho \\ \gamma_0 b\pa_y \rho
\end{pmatrix}\right),
\end{align}
where
\[M=\begin{pmatrix}
(1-\rho)r & 0 & 0 & 0\\
0 & 2\gamma_0(1-\rho)r+2\gamma (1-\rho)rb & 0 & 2\gamma (1-\rho)rb \\
0 & 0 & (1-\rho)b & 0\\
0 & 2\gamma (1-\rho)rb & 0 & 2\gamma_0(1-\rho)b+2\gamma (1-\rho)rb
\end{pmatrix}.\]

We observe from equation \eqref{system_entropy} that we do not have a gradient flow structure. The additional terms result from the different structure of the second order terms in \eqref{system_rb}. They are either of the form $r(1-\rho)$, $b(1-\rho)$ or $rb(1-\rho)$, which correspond to different entropies. This lack of structure results in the different prefactor in \eqref{entropy}.\\
Note that the entropy functional (\ref{entropy}) is also not an entropy in the classical sense as we cannot ensure that it is non-increasing. Nevertheless the entropy grows at most linearly in time, which is sufficient for proving existence of global weak solutions.  
\medskip
\begin{lemma}\label{lemma_entropy}
Let $r, b :\Omega \rightarrow \mathbb{R}^2$ be a sufficiently smooth solution to system \eqref{system_rb} satisfying 
\begin{align*}
0\leq r,b \quad \text{ and }\quad \rho\leq 1.
\end{align*}
Then there exists a constant $C\geq 0$ such that
\begin{align}
\label{entropyinequality}
\begin{aligned}
\frac{\mathrm{d}\mathcal{E}}{\mathrm{d}t}+\mathcal{D}_0&\leq C,
\end{aligned}
\end{align}
where 
\begin{equation*}
\mathcal{D}_0=C_0\int_{\Omega} (1-\rho)|\nabla\sqrt{r}|^2 +(1-\rho)|\nabla\sqrt{b}|^2+ |\nabla\sqrt{1-\rho}|^2+|\nabla \rho|^2\,dx\,dy,
\end{equation*}
for some constant $C_0>0$.
\end{lemma}
\smallskip
\begin{proof}
System (\ref{system_entropy}) enables us to deduce the entropy dissipation relation:
\begin{align}\label{equ_1}
\begin{aligned}
\frac{\mathrm{d}\mathcal{E}}{\mathrm{d}t} &=\int_{\Omega} (u\, \pa_t r +v\, \pa_t b ) dx\,dy\\
&=-\frac{h}{2}\int_{\Omega}  M \begin{pmatrix}
\nabla u \\ \nabla v
\end{pmatrix}\cdot \begin{pmatrix}
\nabla u \\ \nabla v
\end{pmatrix}+\begin{pmatrix}
\frac{r}{2}\pa_x \rho \\ \gamma_0 r \pa_y \rho \\ \frac{b}{2}\pa_x \rho \\ \gamma_0 \pa_y \rho
\end{pmatrix}\cdot \begin{pmatrix}
\nabla u \\ \nabla v
\end{pmatrix}\,dx\,dy\\
&=-\frac{h}{2}\int_{\Omega} \left[ (1-\rho)(r(\pa_x u)^2 +b(\pa_x v)^2) +\frac{1}{2}\pa_x \rho (r \pa_x u+b \pa_x v)\right.\\
&\qquad\qquad +2\gamma_0\left((1-\rho)(r(\pa_y u)^2 +b(\pa_y v)^2) +\frac{1}{2}\pa_y \rho (r \pa_y u+b \pa_y v)\right)\\
&\qquad\qquad +2\gamma(1-\rho)rb((\pa_y u+\pa_y v)^2\bigg]\,dx\,dy,
\end{aligned}
\end{align}
where we have used integration by parts.
For the non-quadratic term in $x$-direction, we use the fact that
\[\pa_x u=\frac{\pa_x r}{r}+\frac{\pa_x \rho}{2(1-\rho)}-\frac{2}{h},\qquad \pa_x v=\frac{\pa_x b}{b}+\frac{\pa_x \rho}{2(1-\rho)}+\frac{2}{h},\]
and deduce that
\begin{align}\label{equ_2}
\begin{aligned}
-\frac{1}{2}\int_{\Omega} \pa_x \rho( r \pa_x u+ b \pa_x v)\, dx\,dy&=-\frac{1}{2}\int_{\Omega} \left(1+\frac{\rho}{2(1-\rho)}\right)(\pa_x \rho)^2\,dx\,dy\\
&\quad+\frac{1}{h}\int_{\Omega} \pa_x \rho (r-b)\, dx\,dy.
\end{aligned}
\end{align}
The first term on the right hand side is negative, for the second we derive that
\begin{align*}
\frac{1}{h}\int_{\Omega} \pa_x \rho (r-b)\,dx\,dy&=\frac{1}{h}\int_{\Omega} (1-\rho)(\pa_x r-\pa_x b)\\
&=\frac{1}{h}\int_{\Omega} (1-\rho)\left(r\pa_x u-b\pa_x v-\frac{r-b}{2(1-\rho)}\pa_x \rho +\frac{2}{h}\rho\right)\,dx\,dy\\
&=\frac{1}{h}\int_{\Omega} (1-\rho)(r\pa_x u-b\pa_x v)\,dx\,dy-\frac{1}{2h}\int \pa_x \rho(r-b)\,dx\,dy\\
&\quad +\frac{2}{h^2}\int \rho (1-\rho)\,dx\,dy.
\end{align*}
Therefore we obtain
\begin{align*}
 \frac{3}{2h}\int_{\Omega} \pa_x \rho (r-b)\,dx\,dy &=\frac{1}{h}\int_{\Omega} (1-\rho)(r\pa_x u-b\pa_x v)\,dx\,dy+\frac{2}{h^2}\int_{\Omega} \rho (1-\rho)\,dx\,dy.
\end{align*}
As $0\leq \rho \leq 1$ and the integration is over a bounded domain, there is a positive constant $\hat{C}$ such that
\begin{equation*}
\frac{2}{h^2}\int_{\Omega} \rho (1-\rho)\,dx\,dy \leq \hat{C}.
\end{equation*}
Applying Young's inequality, we get
\begin{align*}
\int_{\Omega} (1-\rho)r \pa_x u \,dx\,dy&\leq \frac{1}{h}\int_{\Omega} (1-\rho)r \,dx\,dy+\frac{h}{4}\int_{\Omega} (1-\rho)r(\pa_x u)^2 \,dx\,dy
\end{align*}
and therefore
\begin{align*}
\int_{\Omega} (1-\rho)(r \pa_x u- b\pa_x v)\,dx\,dy&\leq \, \frac{1}{h}\int_{\Omega} \rho (1-\rho) \,dx\,dy\\
&\quad +\frac{h}{4}\int_{\Omega} (1-\rho)(r(\pa_x u)^2+b(\pa_x v)^2) \,dx\,dy.
\end{align*}
Altogether we deduce the following estimate from \eqref{equ_2}:
\begin{align*}
-\frac{1}{2}\int_{\Omega} \pa_x \rho( r \pa_x u+ b \pa_x v)\, dx\,dy&\leq -\frac{1}{2}\int_{\Omega} \left(1+\frac{\rho}{2(1-\rho)}\right)(\pa_x \rho)^2\,dx\,dy\\
&\quad +\frac{1}{6}\int_{\Omega} (1-\rho)(r(\pa_x u)^2+b (\pa_x v)^2)\,dx\,dy\\
&\quad +\frac{2}{h^2}\int_{\Omega} \rho (1-\rho)\,dx\,dy.
\end{align*}
We use the same arguments for the term $-\frac{1}{2}\int_{\Omega} \pa_y \rho( r \pa_y u+ b \pa_y v)\, dx\,dy$ and obtain the following entropy dissipation from \eqref{equ_1}: 
\begin{align*}
\begin{aligned}
\frac{\mathrm{d}\mathcal{E}}{\mathrm{d}t} &=\int_{\Omega} (u\, \pa_t r +v\, \pa_t b )  dx\,dy\\
&\leq   -\frac{h}{2}\int_{\Omega} \left[\frac{5}{6}(1-\rho)(r(\pa_x u)^2+b (\pa_x v)^2)\right.+\frac{1}{2} \left(1+\frac{\rho}{2(1-\rho)}\right)\left((\pa_x \rho)^2+2\gamma_0(\pa_y \rho)^2\right)\\
&\qquad\qquad -(1+2\gamma_0)\frac{2}{h^2} \rho (1-\rho)+\frac{5}{3}\gamma_0(1-\rho)(r(\pa_y u)^2+b (\pa_y v)^2)\\
&\qquad\qquad +2\gamma(1-\rho)rb((\pa_y u+\pa_y v)^2\bigg]\,dx\,dy\leq \tilde{C}.
\end{aligned}
\end{align*}
For the analysis it will be sufficient to use a reduced version of the entropy inequality, given by
\begin{align*}
\begin{aligned}
\frac{\mathrm{d}\mathcal{E}}{\mathrm{d}t} &\leq   -C_0\int_{\Omega} (1-\rho)(r|\nabla u|^2+b |\nabla v|^2)+ |\nabla \rho|^2\,dx\,dy+\tilde{C}=:\tilde{\mathcal{D}}_0,
\end{aligned}
\end{align*}
where $C_0:=\frac{h}{2}\min(\frac{1}{2},\gamma_0)$. Using the definitions of $u$ and $v$, applying Young's inequality to estimate the mixed terms as well as the fact that
\begin{align*}
&r(1-\rho)\left|\nabla\left(\log\frac{r}{2(1-\rho)}\right)\right|^2+b(1-\rho)\left|\nabla\left(\log\frac{b}{2(1-\rho)}\right)\right|^2\\
=\,&4(1-\rho)\left|\nabla\sqrt{r}\right|^2+4(1-\rho)\left|\nabla\sqrt{b}\right|^2+4\rho\left|\nabla\sqrt{1-\rho}\right|^2+2\left|\nabla\rho\right|^2,
\end{align*}
we obtain 
\begin{align*}
\tilde{\mathcal{D}}_0&\leq-C_0\int_{\Omega} 2(1-\rho)|\nabla\sqrt{r}|^2 +2(1-\rho)|\nabla\sqrt{b}|^2+2\rho |\nabla\sqrt{1-\rho}|^2+|\nabla \rho|^2\,dx\,dy\nonumber\\
&\quad +\frac{4}{h^2}C_0\int_{\Omega} (1-\rho)(r|\nabla V_r|^2 +b|\nabla V_b|^2)\,dx\,dy-C_0\int_{\Omega} |\nabla \rho|^2\,dx\,dy+\tilde{C}\nonumber.
\end{align*}
Since $|\nabla V_r|^2=|\nabla V_b|^2=1$
and
$2\rho |\nabla\sqrt{1-\rho}|^2+|\nabla \rho|^2\geq 2|\nabla\sqrt{1-\rho}|^2,$
we get the estimate
\begin{align}
\label{entropy_ineq2}
\begin{aligned}
\frac{\mathrm{d}\mathcal{E}}{\mathrm{d}t}&\leq -C_0\int_{\Omega} (1-\rho)(|\nabla\sqrt{r}|^2 +|\nabla\sqrt{b}|^2)+|\nabla\sqrt{1-\rho}|^2+|\nabla \rho|^2\,dx\,dy+C
\end{aligned}
\end{align}
for some constant $C\geq 0$, which concludes the proof.
\end{proof}

\subsection{Positivity}\label{positivity}
We want the global weak solution of system \eqref{system_rb} to satisfy $0\leq r(t),\,b(t),\,\rho(t)\leq 1$ for all $t>0$, if the latter condition is prescribed for the initial data. System \eqref{system_rb} can be written in the form
\begin{align*}
\begin{aligned}
\begin{pmatrix}
\pa_t r\\ \pa_t b
\end{pmatrix} =\begin{pmatrix}
\nabla & 0\\0 & \nabla
\end{pmatrix}\cdot\left( A(r,b) \begin{pmatrix}
\nabla r\\\nabla b
\end{pmatrix}+\begin{pmatrix}
(1-\rho)r\nabla V_r\\(1-\rho)b\nabla V_b
\end{pmatrix}\right),
\end{aligned}
\end{align*}  
where $A=A(r,b)$ is the diffusion matrix given by
\[A(r,b)=\frac{h}{2} \begin{pmatrix}
(1-b) & 0 & r & 0\\
0 & 2\gamma(1-b)b+2\gamma_0(1-b) & 0 &2\gamma(1-r)r+2\gamma_0 r\\
b & 0 & (1-r) & 0\\
0 & 2\gamma(1-b)b+2\gamma_0 b & 0 & 2\gamma(1-r)r+2\gamma_0(1-r)\\
\end{pmatrix}.\]

Note that the diffusion matrix is neither symmetric nor positive definite in general. 
Hence, we cannot use the maximum principle to prove nonnegativity and boundedness of $r$, $b$ and $\rho$. However, the system allows to use a more direct approach to deduce upper and lower bounds for the variables $r$, $b$ and $\rho$. 
We therefore consider the entropy density
\[E:\mathcal{M}\to\mathbb{R},\begin{pmatrix}
r\\b
\end{pmatrix}\mapsto r(\log r-1) + b (\log b-1) + \frac{1}{2} (1-\rho)(\log (1-\rho)-1)+\frac{2}{h} r V_r+\frac{2}{h} bV_b,\] 
where 
\begin{equation}\label{equ:set}
\mathcal{M}=\left\{\begin{pmatrix}
r\\b
\end{pmatrix}\in \mathbb{R}^2:r>0,b>0,r+b<1\right\}.
\end{equation}
\smallskip
\begin{lemma}\label{lemma1}
The function $E:\mathcal{M}\to \mathbb{R}$ is strictly convex and belongs to $C^2(\mathcal{M})$. Its gradient $DE:\mathcal{M}\to \mathbb{R}^2$ is invertible and the inverse of the Hessian $D^2 E:\mathcal{M}\to \mathbb{R}^{2\times 2}$ is uniformly bounded.
\end{lemma}

\begin{proof}
The invertibility of $D E$ can be shown directly. Using the definitions of the entropy variables $u$ and $v$, we get
\[u-v=\log\frac{r}{b}-\frac{4x}{h}\quad \text{ and }\quad u+v=\log \frac{rb}{1-\rho}.\]
Solving these relations for $r$ gives 
\[r=be^{\frac{4x}{h}}e^{u-v}\quad \text{ and }\quad r=\frac{(1-b)e^{u+v}}{b+e^{u+v}},\]
which leads to a quadratic equation in $b$ with exactly one positive solution
\[b=b(u,v)=-\frac{1}{2}\left(e^{u+v}+e^{-\frac{4x}{h}}e^{2v}\right)+\left(\frac{\left(e^{u+v}+e^{-\frac{4x}{h}}e^{2v}\right)^2}{4}+e^{-\frac{4x}{h}}e^{2v}\right)^\frac{1}{2},\]
and therefore
\[r=r(u,v)=-\frac{1}{2}\left(e^{u+v}+e^{\frac{4x}{h}}e^{2u}\right)+\left(\frac{\left(e^{u+v}+e^{\frac{4x}{h}}e^{2u}\right)^2}{4}+e^{\frac{4x}{h}}e^{2u}\right)^\frac{1}{2}.\]
Simple calculations ensure that $\begin{pmatrix}
r\\b
\end{pmatrix}\in\mathcal{M}$.\\
To show the uniform boundedness of the inverse of $D^2 E:\mathcal{M}\to \mathbb{R}^{2\times2}$, we observe that 
\[D E\begin{pmatrix}
r\\b
\end{pmatrix}=\begin{pmatrix}
u\\v
\end{pmatrix} \quad\text{ and }\quad D^2 E\begin{pmatrix}
r\\b
\end{pmatrix}=\begin{pmatrix}
\pa_r u & \pa_b u\\\pa_r v&\pa_b v
\end{pmatrix}=\begin{pmatrix}
\frac{1}{r}+\frac{1}{2(1-\rho)} & \frac{1}{2(1-\rho)}\\\frac{1}{2(1-\rho)}&\frac{1}{b}+\frac{1}{2(1-\rho)}
\end{pmatrix}.\]
Since $0<r,b,\rho< 1$, we can deduce that the inverse of $D^2 E$ exists and is bounded in $\mathcal{M}$.
\end{proof}

\noindent Hence, Lemma \ref{lemma1} ensures that if there exists a weak solution $(u,v) \in L^2(0,T;H^1(\Omega,\mathbb{R}^2))$ to \eqref{system3}, the original variables $\begin{pmatrix}
r\\b
\end{pmatrix}= (DE)^{-1}\begin{pmatrix}
u\\v
\end{pmatrix}$ satisfy $\begin{pmatrix}
r(\cdot,\cdot,t)\\b(\cdot,\cdot,t)
\end{pmatrix} \in\mathcal{M}$ for $t>0$ almost everywhere. This gives us $L^{\infty}$-bounds, necessary for the global in time existence proof in the following section.

\section{Main result}\label{mainresult}
We start this section by stating the notion of weak solutions to system \eqref{system_rb}.
\medskip
\begin{definition}
A function  $(r,b):\Omega\times (0,T)\to \overline{\mathcal{M}}$ is called a weak solution to system \eqref{system_rb} if it satisfies the formulation 
\begin{align}\label{theorem1_1}
\int_0^T \begin{pmatrix}
\pa_t r \\ \pa_t b
\end{pmatrix}\cdot\begin{pmatrix}
\Phi_1\\\Phi_2
\end{pmatrix}\,dt&+\frac{h}{2}\int_0^T\int_{\Omega}
\begin{pmatrix}
\pa_x r(1-\rho)+r \pa_x \rho\\ \pa_x b (1-\rho)+b\pa_x \rho
\end{pmatrix}\cdot\begin{pmatrix}
\pa_x \Phi_1\\ \pa_x \Phi_2
\end{pmatrix}\,dx\,dy\,dt\nonumber\\
&+h\int_0^T \int_{\Omega}\gamma_0\begin{pmatrix}
\pa_y r(1-\rho)+r\pa_y \rho\\ \pa_y b(1-\rho)+b\pa_y \rho
\end{pmatrix}\cdot\begin{pmatrix}
\pa_y \Phi_1\\ \pa_y \Phi_2
\end{pmatrix}\,dx\,dy\,dt\\
&+h\int_0^T \int_{\Omega}\gamma\begin{pmatrix}
\pa_y (r b)(1-\rho)+r b\pa_y \rho\\ \pa_y (r b)(1-\rho)+r b\pa_y \rho
\end{pmatrix}\cdot\begin{pmatrix}
\pa_y \Phi_1\\ \pa_y \Phi_2
\end{pmatrix}\,dx\,dy\,dt\nonumber\\
&+\int_0^T\int_{\Omega}\begin{pmatrix}
(1-\rho)r\nabla V_r\\(1-\rho)b\nabla V_b
\end{pmatrix}\cdot\begin{pmatrix}
\nabla \Phi_1\\ \nabla \Phi_2
\end{pmatrix}\,dx\,dy\,dt=0,\nonumber
\end{align}
for all $\Phi_1,\Phi_2\in L^2(0,T;H^1(\Omega))$. 
\end{definition}
\medskip
\begin{theorem}{(Global existence)}\label{theorem1}
Let $T>0$, let $(r_0,b_0):\Omega \to \mathcal{M}$, where $\mathcal{M}$ is defined by \eqref{equ:set}, be a measurable function such that $E(r_0,b_0)\in L^1(\Omega)$. Then there exists a weak solution $(r,b):\Omega\times (0,T)\to \overline{\mathcal{M}}$ in the sense of \eqref{theorem1_1} with periodic boundary conditions in $x$-direction and no-flux boundary conditions in $y$-direction satisfying 
\begin{align*}
&\pa_t r,\, \pa_t b \in L^2(0,T;H^1(\Omega)'),\\
&\rho,\, \sqrt{1-\rho}\,\in L^2(0,T;H^1(\Omega )),\\
&(1-\rho)\nabla\sqrt{r},\,(1-\rho)\nabla\sqrt{b} \,\in L^2(0,T;L^2(\Omega)).
\end{align*}
Moreover, the weak solution satisfies the following entropy dissipation inequality:
\begin{align}
\label{theorem1_2}
\begin{aligned}
\frac{\mathrm{d}\mathcal{E}}{\mathrm{d}t} +\mathcal{D}_1\leq C,
\end{aligned}
\end{align}
where
\begin{equation*}
\mathcal{D}_1=C_0\int_{\Omega} (1-\rho)^2|\nabla\sqrt{r}|^2 +(1-\rho)^2|\nabla\sqrt{b}|^2+ |\nabla\sqrt{1-\rho}|^2+|\nabla \rho|^2\,dx\,dy
\end{equation*}
and $C_0$ and $C$ are the constants from (\ref{entropy_ineq2}).
\end{theorem}
\medskip
We would like to mention the different dissipation term in \eqref{theorem1_2}. In particular, since the convergence properties are not strong enough to pass to the limit in the entropy dissipation \eqref{entropyinequality}, we obtain a modified entropy inequality \eqref{theorem1_2}.\\ A major difference in the analysis of the system compared to related ones in the literature (cf. \cite{zamponi2015analysis}, \cite{MR2745794}, \cite{schlake2011mathematical}, \cite{MR2801186}) is the fact that we have an anisotropic diffusion and no gradient flow structure, which requires a different entropy and a priori estimates.\\

The following existence proof is based on an approximation of \eqref{system_rb}. The basis of the approximation argument is the following formulation of system \eqref{system_rb}:

\begin{align} \label{system3}
\begin{aligned}
\begin{pmatrix}
\pa_t r \\ \pa_t b
\end{pmatrix}&=\begin{pmatrix}
\nabla &0\\0&\nabla
\end{pmatrix}\cdot\left( G(r,b)\begin{pmatrix}
\nabla u\\ \nabla v
\end{pmatrix}
+ H(r,b)\right),
\end{aligned}
\end{align} 
where 
\[G=\frac{h}{2}\begin{pmatrix}
(1-\rho)r(1+\frac{1}{2-\rho}r) & 0 & \frac{(1-\rho)rb}{2-\rho} & 0\\
0 & 2(1-\rho)r(\gamma_0(\frac{2-b}{2-\rho})+\gamma b) & 0 & 2(1-\rho)rb(\frac{\gamma_0}{2-\rho}+\gamma)\\
(1-\rho)b(1+\frac{1}{2-\rho}r) & 0 & \frac{(1-\rho)rb}{2-\rho} & 0\\
0 & 2(1-\rho)b(\gamma_0(\frac{2-b}{2-\rho})+\gamma r) & 0 & 2(1-\rho)b(\frac{\gamma_0 b}{2-\rho}+\gamma r)\\
\end{pmatrix}\]
and
\[H=\begin{pmatrix}
(1-\rho)r\frac{r-b}{2-\rho}\\ 0\\ (1-\rho)b\frac{r-b}{2-\rho}\\0
\end{pmatrix}.\]
The positive semi-definiteness of the matrix $G(r,b)$ can be proven using a similar approach as we have seen in Subsection \ref{entropy_disspation}.\\ 
We discretize system \eqref{system3} in time using the implicit Euler scheme with time step $\tau>0$ which results in a recursive sequence of elliptic problems. These are modified by adding higher order regularization terms. The corresponding weak formulation is given by
\begin{align}\label{equ_3}
\begin{aligned}
\frac{1}{\tau}\begin{pmatrix}
r_k-r_{k-1} \\ b_k-b_{k-1}
\end{pmatrix}&=\begin{pmatrix}
\nabla &0\\0&\nabla
\end{pmatrix}\cdot\left( G(r_k,b_k)\begin{pmatrix}
\nabla u_k\\ \nabla v_k
\end{pmatrix}+ H(r_k,b_k)\right)\\
&\quad+\tau\begin{pmatrix}
\Delta u_k+u_k\\\Delta v_k+v_k
\end{pmatrix}.
\end{aligned}
\end{align} 
The regularization guarantees coercivity of the elliptic system in $H^1(\Omega)$. This is in contrast to \cite{zamponi2015analysis} who used a stronger regularization by introducing a Bilaplacian.
The existence proof is divided into several steps. First we show existence of weak solutions to the regularized, discrete in time problem by applying Lax-Milgram to a linearized version of the problem \eqref{equ_3} and using the Schauder fixed point theorem to conclude the existence result for the corresponding nonlinear problem.\\
Finally uniform a priori estimates in $\tau$ and the use of a generalized Aubin-Lions lemma (cf. \cite{zamponi2015analysis}) allow to pass to the limit $\tau \to 0$. Note that one can also use the Kolmogorov-Riesz theorem in a similar fashion to \cite{MR2745794}. \\

\subsection{Time discretization and regularization of system \eqref{system_rb}}
We start by studying the regularized time discrete system. Recall that the entropy variables are defined as $(u,v)=DE(r,b)$ for $(r,b)\in \mathcal{M}$. Lemma \ref{lemma1} ensures that $DE$ is invertible, hence we set $(r,b)=(DE)^{-1}(u,v)$ for $(u,v)\in \mathbb{R}^2$.\\
Let $T>0$, $N\in\mathbb{N}$ and let $\tau=T/N$ be the time step size. We split the time interval into the subintervals
\[(0,T]=\bigcup_{k=1}^N ((k-1)\tau,k\tau],\qquad \tau=\frac{T}{N}.\]
Then for given functions $(r_{k-1}, b_{k-1}) \in \overline{\mathcal{M}}$, which approximate $(r,b)$ at time $\tau(k-1)$, we want to find $(r_k,b_k) \in \overline{\mathcal{M}}$ solving the regularized time discrete problem \eqref{equ_3} in the weak formulation: 
\begin{align}\label{time_discrete_weak}
\begin{aligned}
\frac{1}{\tau}\int_{\Omega}\begin{pmatrix}
r_k-r_{k-1}\\b_k-b_{k-1}
\end{pmatrix}&\cdot\begin{pmatrix}
\Phi_1\\\Phi_2
\end{pmatrix}\,dx\,dy+\int_{\Omega}\begin{pmatrix}
\nabla \Phi_1\\ \nabla \Phi_2
\end{pmatrix}^T G(r_k,b_k)\begin{pmatrix}
\nabla u_k \\ \nabla v_k 
\end{pmatrix}\,dx\,dy\\
+&\int_{\Omega} H(r_k,b_k)\begin{pmatrix}
\nabla \Phi_1\\ \nabla \Phi_2
\end{pmatrix}\,dx\,dy+\tau R\left(\begin{pmatrix}
\Phi_1\\ \Phi_2
\end{pmatrix},\begin{pmatrix}
u_k\\v_k
\end{pmatrix}\right)=0
\end{aligned}
\end{align}
for $(\Phi_1,\Phi_2)\in H^1(\Omega)\times H^1(\Omega)$, where $(r_k,b_k)=DE^{-1}(u_k,v_k)$ and
\begin{align*}
\begin{aligned}
R\left(\begin{pmatrix}
\Phi_1\\ \Phi_2
\end{pmatrix},\begin{pmatrix}
u_k\\v_k
\end{pmatrix}\right)&=\int_{\Omega}
\Phi_1u_k+\Phi_2v_k+\nabla \Phi_1\cdot \nabla u_k+\nabla \Phi_2\cdot \nabla v_k
\,dx\,dy.
\end{aligned}
\end{align*}

Note that it is not immediately evident that we can apply the transformation from Lemma \ref{lemma1} to $(u_k,v_k)$, since the transformation $(u,v)=DE(r,b)$ is only defined for $(r,b)\in \mathcal{M}$ and $(r,b)=DE^{-1}(u,v)$ for $(u,v)\in \mathbb{R}^2$ respectively. Since we only know that $(u,v)\in L^2(\Omega,\mathbb{R}^2)$, we do not have uniform boundedness. However, $u,v$ take values $\pm \infty$ at most on a set of measure zero. Hence we know that $(r,b) \in \mathcal{M}$ a.e., which allows us to apply the variable transformation.\\

We define $\mathcal{F}:\overline{\mathcal{M}}\subseteq L^2(\Omega,\mathbb{R}^2)\to\overline{\mathcal{M}}\subseteq L^2(\Omega,\mathbb{R}^2), (\tilde{r},\tilde{b}) \mapsto (r,b)=DE^{-1}(u,v)$, where
$(u,v)$ is the unique solution in $H^1(\Omega,\mathbb{R}^2)$ to the linear problem
\begin{equation} \label{equ1} 
a((u,v),(\Phi_1,\Phi_2))=F(\Phi_1,\Phi_2) \quad \text{for all }(\Phi_1,\Phi_2)\in H^1(\Omega,\mathbb{R}^2)
\end{equation}
with 
\begin{align*}
&a((u,v),(\Phi_1,\Phi_2))=\int_{\Omega}\begin{pmatrix}
\nabla \Phi_1\\ \nabla \Phi_2
\end{pmatrix}^T G(\tilde{r},\tilde{b})\begin{pmatrix}
\nabla u \\ \nabla v 
\end{pmatrix}\,dx\,dy+\tau R\left(\begin{pmatrix}
\Phi_1\\ \Phi_2
\end{pmatrix},\begin{pmatrix}
u\\v
\end{pmatrix}\right)\\
&F(\Phi_1,\Phi_2)=-\frac{1}{\tau}\int_{\Omega}\begin{pmatrix}
\tilde{r}-r_{k-1}\\\tilde{b}-b_{k-1}
\end{pmatrix}\cdot\begin{pmatrix}
\Phi_1\\\Phi_2
\end{pmatrix}\,dx\,dy+ \int_{\Omega}H(\tilde{r},\tilde{b})\begin{pmatrix}
\nabla \Phi_1\\ \nabla \Phi_2
\end{pmatrix}\,dx\,dy.
\end{align*}
The bilinear form $a:H^1(\Omega;\mathbb{R}^2)\times H^1(\Omega;\mathbb{R}^2)\to \mathbb{R}$ and the functional $F:H^1(\Omega,\mathbb{R}^2)\to \mathbb{R}$ are bounded. Moreover, $a$ is coercive since the positive semi-definiteness of $G(r,b)$ implies that
\begin{align*}
a((u,v),(u,v))&=\int_{\Omega}\begin{pmatrix}
\nabla u\\ \nabla v
\end{pmatrix}^T G(\tilde{r},\tilde{b})\begin{pmatrix}
\nabla u \\ \nabla v 
\end{pmatrix}\,dx\,dy+\tau R\left(\begin{pmatrix}
u\\ v
\end{pmatrix},\begin{pmatrix}
u\\v
\end{pmatrix}\right)\\
&\geq \tau \left(\|u\|_{H^1(\Omega)}^2+\|v\|_{H^1(\Omega)}^2\right).
\end{align*}
Then the Lax-Milgram lemma guarantees the existence of a unique solution $(u,v)\in H^1(\Omega;\mathbb{R}^2)$ to \eqref{equ1}.\\
To apply Schauer's fixed point theorem, we need to show that $\mathcal{F}$ is continuous.
Therefore, let $(\tilde{r}_k,\tilde{b}_k)$ be a sequence in $\overline{\mathcal{M}}$ converging strongly to $(\tilde{r},\tilde{b})$ in $L^2(\Omega,\mathbb{R}^2)$ and let $(u_k,v_k)$ be the corresponding unique solution to \eqref{equ1} in $H^1(\Omega;\mathbb{R}^2)$. We have that $G(\tilde{r}_k,\tilde{b}_k)\to G(\tilde{r},\tilde{b})$ and $H(\tilde{r}_k,\tilde{b}_k)\to H(\tilde{r},\tilde{b})$ strongly in $L^2(\Omega,\mathbb{R}^2)$. As the entropy inequality yields a uniform bound for $(u_k,v_k)$ in $H^1(\Omega;\mathbb{R}^2)$, there exists a subsequence with $(u_k,v_k)\rightharpoonup (u,v)$ weakly in $H^1(\Omega;\mathbb{R}^2)$. In order to identify $(u,v)$ as the solution of \eqref{equ1} with coefficients $(\tilde{r},\tilde{b})$, we first consider problem \eqref{equ1} only for test functions in $(\Phi_1,\Phi_2)\in W^{1,\infty}(\Omega,\mathbb{R}^2)$. Here, the (weak) limit $(u,v)$ is well defined. Then, the $L^{\infty}$ bounds of $G(\tilde{r}_k,\tilde{b}_k)$ allow us to consider the problem \eqref{equ1} for all $(\Phi_1,\Phi_2)\in H^1(\Omega,\mathbb{R}^2)$ applying a density argument. So, the limit $(u,v)$ as the solution of problem \eqref{equ1} with coefficients $(\tilde{r},\tilde{b})$ is well defined.\\
In view of the compact embedding $H^1(\Omega,\mathbb{R}^2)\hookrightarrow L^2(\Omega,\mathbb{R}^2)$, we have a subsequence (not relabeled) with $(u_k,v_k)\to (u,v)$ strongly in $L^2(\Omega,\mathbb{R}^2)$. Since the limit is unique, the whole sequence converges. Together with the property that the map from $(u,v)$ to $(r,b)$ is Lipschitz continuous (cf. Lemma \ref{lemma1}), we have continuity of $\mathcal{F}$.\\
Furthermore, the compact embedding $H^1(\Omega;\mathbb{R}^2) \hookrightarrow L^2(\Omega;\mathbb{R}^2)$ gives the compactness of $\mathcal{F}$. Combined with the property that $\mathcal{F}$ maps a convex, closed set onto itself, we can apply Schauder's fixed point theorem, which assures the existence of a solution $(r,b)\in \overline{\mathcal{M}}$ to  \eqref{equ1} with $(\tilde{r},\tilde{b})$ replaced by $(r,b)$.

The convexity of $E$ implies that $E(\varphi_1)-E(\varphi_2)\leq DE(\varphi_1)\cdot(\varphi_1-\varphi_2)$ for all $\varphi_1,\varphi_2\in\mathcal{M}$. Choosing $\varphi_1=(r_k,b_k)$ and $\varphi_2=(r_{k-1},b_{k-1})$ and using $DE(r_k,b_k)=(u_k,v_k)$, we obtain
\begin{align}\label{equ13}
\frac{1}{\tau}\int_{\Omega}&\begin{pmatrix}
r_k-r_{k-1}\\b_k-b_{k-1}
\end{pmatrix}\cdot\begin{pmatrix}
u_k\\v_k
\end{pmatrix}\,dx\,dy\geq\frac{1}{\tau}\int_{\Omega}\begin{pmatrix}
E(r_k,b_k)-E(r_{k-1},b_{k-1})
\end{pmatrix}\,dx\,dy.
\end{align}
\noindent Employing the test function $(\Phi_1,\Phi_2)=(u_k,v_k)$ in \eqref{time_discrete_weak} and applying \eqref{equ13}, we obtain
\begin{align}\label{time_entropy}
\begin{aligned}
&\quad\int_{\Omega}E(r_k,b_k)\,dx\,dy+\tau\int_{\Omega}\begin{pmatrix}
\nabla u_k\\ \nabla v_k
\end{pmatrix}^T G(r_k,b_k)\begin{pmatrix}
\nabla u_k \\ \nabla v_k 
\end{pmatrix}\,dx\,dy \\
&+\tau \int_{\Omega}  H(r_k,b_k)\begin{pmatrix}
\nabla u_k\\ \nabla v_k
\end{pmatrix}\,dx\,dy+\tau^2 R \left(\begin{pmatrix}
u_k\\v_k
\end{pmatrix},\begin{pmatrix}
u_k\\v_k
\end{pmatrix}\right)\leq\int_{\Omega}E(r_{k-1},b_{k-1})\,dx\,dy.
\end{aligned}
\end{align}
Using the entropy inequality \eqref{entropy_ineq2}, resolving recursion \eqref{time_entropy} leads to
\begin{align}\label{discrete_entropy_ineq}
\begin{aligned}
&\quad\int_{\Omega} E(r_k,b_k)\,dx\,dy+C_0\tau\sum_{j=1}^k\int_{\Omega} (1-\rho_j)|\nabla\sqrt{r_j}|^2\,dx\,dy\\ 
&+C_0\tau\sum_{j=1}^k\int_{\Omega}(1-\rho_j)|\nabla\sqrt{b_j}|^2+ |\nabla\sqrt{1-\rho_j}|^2+|\nabla \rho_j|^2\,dx\,dy\\
&+\tau^2\sum_{j=1}^k R \left(\begin{pmatrix}
u_j\\v_j
\end{pmatrix},\begin{pmatrix}
u_j\\v_j
\end{pmatrix}\right) \leq \int_{\Omega} E(r_0,b_0)\,dx\,dy+T C.
\end{aligned}
\end{align}

\subsection{The limit $\tau \to 0$}\label{limittau}
Let $(r_k,b_k)$ be a sequence of solutions to \eqref{time_discrete_weak}. We define $r_\tau(x,y,t)=r_k(x,y)$ and $b_\tau(x,y,t)=b_k(x,y)$ for $(x,y)\in\Omega$ and $t\in ((k-1)\tau,k\tau]$. Then $(r_\tau,b_\tau)$ solves the following problem, where $\sigma_\tau$ denotes a shift operator, i.e. $(\sigma_\tau r_\tau)(x,y,t)=r_\tau (x,y,t-\tau)$ and $(\sigma_\tau b_\tau)(x,y,t)=b_\tau (x,y,t-\tau)$ for $\tau \leq t\leq T$,
\begin{align}\label{time_discrete_weak_3}
\begin{aligned}
 \frac{1}{\tau}&\int_0^T\int_{\Omega}\begin{pmatrix}
r_\tau-\sigma_\tau r_\tau\\b_\tau-\sigma_\tau b_\tau
\end{pmatrix}\cdot\begin{pmatrix}
\Phi_1\\\Phi_2
\end{pmatrix}\,dx\,dy\,dt\\
&+\frac{h}{2}\int_0^T\int_{\Omega}
\begin{pmatrix}
\pa_x r_\tau(1-\rho_\tau)+r_\tau \pa_x \rho_\tau\\ \pa_x b_\tau (1-\rho_\tau)+b_\tau\pa_x \rho_\tau
\end{pmatrix}\cdot\begin{pmatrix}
\pa_x \Phi_1\\ \pa_x \Phi_2
\end{pmatrix}\,dx\,dy\,dt\\
&+h\int_0^T \int_{\Omega}\gamma_0\begin{pmatrix}
\pa_y r_\tau(1-\rho_\tau)+r_\tau\pa_y \rho_\tau\\ \pa_y b_\tau(1-\rho_\tau)+b_\tau\pa_y \rho_\tau
\end{pmatrix}\cdot\begin{pmatrix}
\pa_y \Phi_1\\ \pa_y \Phi_2
\end{pmatrix}\,dx\,dy\,dt\\
&+h\int_0^T \int_{\Omega}\gamma\begin{pmatrix}
\pa_y (r_\tau b_\tau)(1-\rho_\tau)+r_\tau b_\tau\pa_y \rho_\tau\\ \pa_y (r_\tau b_\tau)(1-\rho_\tau)+r_\tau b_\tau\pa_y \rho_\tau
\end{pmatrix}\cdot\begin{pmatrix}
\pa_y \Phi_1\\ \pa_y \Phi_2
\end{pmatrix}\,dx\,dy\,dt\\
&+\int_0^T\int_{\Omega}\begin{pmatrix}
(1-\rho_\tau)r_\tau\nabla V_r\\(1-\rho_\tau)b_\tau\nabla V_b
\end{pmatrix}\cdot\begin{pmatrix}
\nabla \Phi_1\\ \nabla \Phi_2
\end{pmatrix}\,dx\,dy+\tau R\left(\begin{pmatrix}
\Phi_1\\\Phi_2
\end{pmatrix},\begin{pmatrix}
u_\tau\\v_\tau
\end{pmatrix}\right)\,dt=0,
\end{aligned}
\end{align}

for $(\Phi_1(t),\Phi_2(t))\in L^2(0,T;H^1(\Omega))$. Inequality \eqref{discrete_entropy_ineq} becomes

\begin{align}\label{entropy_ineq_3}
\begin{aligned}
&\quad\int_{\Omega} E(r_\tau(T),b_\tau(T))\,dx\,dy+C_0\int_0^T\int_{\Omega} (1-\rho_\tau)|\nabla\sqrt{r_\tau}|^2\,dx\,dy\,dt\\
&+C_0\int_0^T\int_{\Omega} (1-\rho_\tau)|\nabla\sqrt{b_\tau}|^2+ |\nabla\sqrt{1-\rho_\tau}|^2+|\nabla \rho_\tau|^2\,dx\,dy\,dt\\
&+\tau \int_0^T R\left(\begin{pmatrix}
u_\tau\\v_\tau
\end{pmatrix},\begin{pmatrix}
u_\tau\\v_\tau
\end{pmatrix}\right)\,dt\leq  \int_{\Omega} E(r_0,b_0)\,dx\,dy+T C.
\end{aligned}
\end{align}

\bigskip
The previous inequalities allow us to deduce the following Lemma. Note that from now on $K$ denotes a generic constant independent of $\tau$.  
\begin{lemma}{(A priori estimates)}\label{lemma4}
There exists a constant $K\in\mathbb{R}^+$, such that the following bounds hold:
\begin{align} \label{apriori2}
\begin{aligned}
\|\sqrt{1-\rho_\tau}\nabla\sqrt{r_\tau}\|_{L^2(0,T;L^2(\Omega))}+\|\sqrt{1-\rho_\tau}\nabla\sqrt{b_\tau}\|_{L^2(0,T;L^2(\Omega))}&\leq K, \\
\|\sqrt{1-\rho_\tau}\|_{L^2(0,T;H^1(\Omega))}+\|\rho_\tau\|_{L^2(0,T;H^1(\Omega))}&\leq K, \\
\sqrt{\tau}(\|u_\tau\|_{L^2(0,T;H^1(\Omega))}+\|v_\tau\|_{L^2(0,T;H^1(\Omega))})&\leq K.
\end{aligned}
\end{align}
\end{lemma}
The bounds in \eqref{apriori2} together with the $L^{\infty}$-bounds for $r_\tau,b_\tau$ and $\rho_\tau$ imply
\begin{align} \label{apriori3}
\begin{aligned}
&\|\nabla r_\tau(1-\rho_\tau)+r_\tau \nabla \rho_\tau\|_{L^2(0,T;L^2(\Omega))}\\
\leq &\, 2\|\sqrt{r_\tau}\|_{L^{\infty}(0,T;L^{\infty}(\Omega))} \|\sqrt{1-\rho_\tau}\|_{L^{\infty}(0,T;L^{\infty}(\Omega))}\|\sqrt{1-\rho_\tau}\nabla\sqrt{r_\tau}\|_{L^2(0,T;L^2(\Omega))}\\
&+\|r_\tau\|_{L^{\infty}(0,T;L^{\infty}(\Omega))}\|\nabla\rho_\tau\|_{L^2(0,T;L^2(\Omega))}\leq K,
\end{aligned}
\end{align}
with an analogous inequality for $b_\tau$.
Similarly, we get the estimate
\begin{align}\label{apriori4}
\|\nabla (r_\tau b_\tau)(1-\rho_\tau)&+r_\tau b_\tau\nabla \rho_\tau\|_{L^2(0,T;L^2(\Omega))}\leq K.
\end{align}
For applying Aubin's lemma, we need one more property involving the time derivatives of $r_\tau$ and $b_\tau$.
\begin{lemma} \label{lemma5}
The discrete time derivatives of $r_\tau$ and $b_\tau$ are uniformly bounded, i.e.
\begin{align}\label{equ14}
\frac{1}{\tau}\|r_\tau-\sigma_\tau r_\tau\|_{L^2(0,T;H^1(\Omega)')}+\frac{1}{\tau}\|b_\tau-\sigma_\tau b_\tau\|_{L^2(0,T;H^1(\Omega)')}&\leq K. 
\end{align}
\end{lemma}
\begin{proof}
Let $\Phi \in L^2(0,T;H^1(\Omega))$. Using the estimates in \eqref{apriori2}, \eqref{apriori3} and \eqref{apriori4} yields
\begin{align*}
&\frac{1}{\tau}\int_0^T \langle r_\tau-\sigma_\tau r_\tau,\Phi\rangle\,dt\\
=\,&-\frac{h}{2}\int_0^T\int_{\Omega}(\pa_x r_\tau(1-\rho_\tau)+r_\tau \pa_x \rho_\tau)\pa_x\Phi\,dx\,dy\,dt\\
&-h\int_0^T \int_{\Omega}(\gamma_0\pa_y r_\tau(1-\rho_\tau)+r_\tau\pa_y \rho_\tau)\pa_y \Phi\,dx\,dy\,dt\\
&-h\int_0^T \int_{\Omega}(\gamma\pa_y (r_\tau b_\tau)(1-\rho_\tau)+r_\tau b_\tau\pa_y \rho_\tau)\pa_y \Phi\,dx\,dy\,dt\\
&-\int_0^T\int_{\Omega}
(1-\rho_\tau)r_\tau\nabla V_r\cdot
\nabla \Phi\,dx\,dy\,dt\\
&-\tau \int_0^T\int_{\Omega} u_\tau\Phi+ \nabla u_\tau\cdot\nabla \Phi\,dx\,dy\,dt\\
\leq \,& \frac{h}{2}\|\pa_x r_\tau(1-\rho_\tau)+r_\tau \pa_x \rho_\tau\|_{L^2(0,T;L^2(\Omega))}\|\pa_x\Phi\|_{L^2(0,T;L^2(\Omega))}\\
&+h\|\gamma_0\pa_y r_\tau(1-\rho_\tau)+r_\tau\pa_y \rho_\tau\|_{L^2(0,T;L^2(\Omega))}\|\pa_y\Phi\|_{L^2(0,T;L^2(\Omega))}\\
&+h\|\gamma\pa_y (r_\tau b_\tau)(1-\rho_\tau)+r_\tau b_\tau\pa_y \rho_\tau\|_{L^2(0,T;L^2(\Omega))}\|\pa_y\Phi\|_{L^2(0,T;L^2(\Omega))}\\
&+\|(1-\rho_\tau)r_\tau\nabla V_r\|_{L^{\infty}(0,T;L^{\infty}(\Omega))}\|\nabla \Phi\|_{L^1(0,T;L^1(\Omega))}\\
&+\tau \|u_\tau\|_{L^2(0,T;H^1(\Omega))}\|\Phi\|_{L^2(0,T;H^1(\Omega))}\\
\leq & \, K\|\Phi\|_{L^2(0,T;H^1(\Omega))}.
\end{align*}
A similar estimate can be deduced for $b$ which concludes the proof.\qquad
\end{proof}
\smallskip
From Lemma \ref{lemma4} and Lemma \ref{lemma5} we know that $\rho_\tau\in L^2(0,T;H^1(\Omega))$ and $\frac{1}{\tau}(\rho_\tau-\sigma_\tau  \rho_\tau)\in L^2(0,T;H^1(\Omega)')$ respectively. This enables us to use Aubin's lemma (cf. \cite{dreher2012compact}, Theorem 1) to conclude the existence of a subsequence, also denoted by $\rho_\tau$, such that, as $\tau\to 0$:
\begin{equation*}
\rho_\tau \to \rho \quad \text{ strongly in } L^2(0,T;L^2(\Omega)).
\end{equation*}
This implies
\begin{align}
1-\rho_\tau \to 1- \rho \quad \text{ strongly in } L^2(0,T;L^2(\Omega)),\label{equ5}\\
\sqrt{1-\rho_\tau} \to \sqrt{1- \rho} \quad \text{ strongly in } L^4(0,T;L^4(\Omega)).\nonumber
\end{align}
Due to the continuous embedding of $L^4(0,T;L^4(\Omega))$ in $L^2(0,T;L^2(\Omega))$, it also holds that
\begin{equation}\label{equ7}
\sqrt{1-\rho_\tau} \to \sqrt{1- \rho} \quad \text{ strongly in } L^2(0,T;L^2(\Omega)).
\end{equation}
To pass to the limit $\tau\to 0$ in \eqref{time_discrete_weak_3}, we need to identify the weak $L^2$- limiting functions of the following terms:\\

\begin{compactenum}[(i)]
\item$\nabla r_\tau(1-\rho_\tau)+r_\tau \nabla \rho_\tau,\quad\nabla b_\tau (1-\rho_\tau)+ b_\tau\nabla \rho_\tau$\label{1}\\
\item$\pa_y (r_\tau b_\tau)(1-\rho_\tau)+r_\tau b_\tau\pa_y \rho_\tau$\label{2}\\ 
\item$(1-\rho_\tau)r_\tau\nabla V_r,\quad (1-\rho_\tau)b_\tau\nabla V_b$\label{3}\\
\item$\tau u_\tau,\,\tau v_\tau,\,\tau \nabla u_\tau,\,\tau \nabla v_\tau$\label{4}\\
\end{compactenum}

The terms in \eqref{3} converge weakly in $L^2(0,T;L^2(\Omega))$ as $1-\rho_\tau$ converges strongly in $L^2(0,T;L^2(\Omega))$ by \eqref{equ5} and because of the $L^{\infty}$ bounds for $b_\tau$ and $r_\tau$, up to a subsequence,
\begin{equation}\label{equ8}
r_\tau\rightharpoonup r, \quad b_\tau\rightharpoonup b \quad \text{ weakly}^*\text{ in } L^{\infty}(0,T;L^{\infty}(\Omega)).
\end{equation} 
Because of \eqref{apriori2}, we get that
\begin{equation*}
\tau u_\tau, \tau v_\tau \to 0 \quad \text{ strongly in } L^2(0,T;H^1(\Omega)),
\end{equation*}
which identifies the limit in \eqref{4}.\\
The weak convergence of the terms in \eqref{1} and \eqref{2} can be shown with the help of a generalized Aubin-Lions Lemma (see Lemma 7 in \cite{zamponi2015analysis}). It states that if \eqref{equ14}, \eqref{equ7}, \eqref{equ8} and 
\begin{equation}\label{equ15}
\|\sqrt{1-\rho_\tau}\,g\|_{L^2(0,T;H^1(\Omega))}\leq K\quad\text{ for } g\in \{1,r_\tau,b_\tau\}
\end{equation} 
hold, we have strong convergence up to a subsequence for all $f=f(r_\tau,b_\tau)\in C^0$ of
\begin{equation}\label{equ9}
\sqrt{1-\rho_\tau}f(r_\tau,b_\tau) \to \sqrt{1-\rho}f(r,b) \quad \text{ strongly in } L^2(0,T;L^2(\Omega)), 
\end{equation}
as $\tau \to 0$. Note that \eqref{equ15} can be deduced from the previous a priori estimates.
Writing \eqref{1} as
\begin{align*}
\nabla r_\tau(1-\rho_\tau)+r_\tau \nabla \rho_\tau&=\sqrt{1-\rho_\tau}\nabla (\sqrt{1-\rho_\tau} r_\tau)-r_\tau\sqrt{1-\rho_\tau}\nabla \sqrt{1-\rho_\tau}-r_\tau \nabla(1- \rho_\tau)\\
&=\sqrt{1-\rho_\tau}\nabla (\sqrt{1-\rho_\tau} r_\tau)-r_\tau\sqrt{1-\rho_\tau}\nabla \sqrt{1-\rho_\tau}\\
&\quad -2r_\tau\sqrt{1-\rho_\tau} \nabla\sqrt{1- \rho_\tau}\\
&=\sqrt{1-\rho_\tau}\nabla (\sqrt{1-\rho_\tau} r_\tau)-3r_\tau\sqrt{1-\rho_\tau}\nabla \sqrt{1-\rho_\tau}
\end{align*}
and applying \eqref{equ9} with $f(r_\tau,b_\tau)=r_\tau$, we get that 
\[\sqrt{1-\rho_\tau}\,r_\tau\to \sqrt{1-\rho}\,r \quad \text{ strongly in }L^2(0,T,L^2(\Omega)).\]
Moreover, the $L^{\infty}$ bounds together with \eqref{apriori2} give us $L^2$ bounds for $\nabla (\sqrt{1-\rho_\tau} r_\tau)=\nabla\sqrt{1-\rho_\tau}r_\tau+2\sqrt{r_\tau}\sqrt{1-\rho_\tau}\nabla\sqrt{r_\tau}$ and $\nabla \sqrt{1-\rho_\tau}$.\\
Together, we have
\begin{equation*}
\nabla r_\tau(1-\rho_\tau) +r_\tau \nabla \rho_\tau\rightharpoonup \nabla r(1-\rho)+r\nabla \rho \quad  \text{ weakly in } L^2(0,T;L^2(\Omega)).
\end{equation*}
Similarly, \eqref{2} can be written as
\begin{align*}
\pa_y (r_\tau b_\tau)(1-\rho_\tau)+r_\tau b_\tau \pa_y \rho_\tau&=\sqrt{1-\rho_\tau}\pa_y (\sqrt{1-\rho_\tau} r_\tau b_\tau)-3r_\tau b_\tau\sqrt{1-\rho_\tau}\pa_y \sqrt{1-\rho_\tau}.
\end{align*}
Applying \eqref{equ9} with $f(r_\tau,b_\tau)=r_\tau b_\tau$ and using analogous arguments as in (\ref{1}), we get that
\begin{equation}\label{equ11}
\pa_y (r_\tau b_\tau) (1-\rho_\tau) +r_\tau b_\tau \pa_y \rho_\tau\rightharpoonup \pa_y (rb)(1-\rho)+rb\pa_y \rho \quad \text{ weakly in } L^2(0,T;L^2(\Omega)).
\end{equation}

From Lemma \ref{lemma5} we derive that
\[\frac{1}{\tau}(r_\tau-\sigma_\tau r_\tau)\rightharpoonup \pa_t r,\quad\frac{1}{\tau}(b_\tau-\sigma_\tau b_\tau)\rightharpoonup \pa_t b \quad \quad \text{ weakly in } L^2(0,T;H^1(\Omega)').\]
This, together with the convergences in \eqref{equ8}-\eqref{equ11} allows to finally pass to the limit $\tau\to 0$ in \eqref{time_discrete_weak_3}, which gives the weak formulation \eqref{theorem1_1}.\\
The only thing which remains to verify is the entropy inequality \eqref{theorem1_2}. Since $E$ is convex and continuous, it is weakly lower semi-continuous. Because of the weak convergence of $(r_\tau(t),b_\tau(t))$,
\[\int_\Omega E(r(t),b(t))\,dx\,dy\leq \liminf_{\tau\to 0}\int_\Omega E(r_\tau (t),b_\tau (t))\,dx\,dy \quad \text{ for a.e. } t>0.\]
We cannot expect the identification of the limit of $\sqrt{1-\rho_\tau}\nabla \sqrt{r_\tau}$, but employing \eqref{equ9} with $f(r,b)=\sqrt{r}$, we get 
\[\sqrt{1-\rho_\tau} \sqrt{r_\tau} \to \sqrt{1-\rho}\sqrt{r} \quad \text{ strongly in }L^2(0,T;L^2(\Omega))\]
with analogous convergence results for $r$ being replaced by $b$. 
Because of the $L^{\infty}$-bounds and the bounds in \eqref{entropy_ineq_3}, we obtain $\nabla(\sqrt{1-\rho_\tau}\sqrt{r_\tau})\in L^2(0,T;L^2(\Omega))$, which implies 
\begin{align}\label{equ12}
\begin{aligned}
\sqrt{1-\rho_\tau} \sqrt{r_\tau} &\rightharpoonup \sqrt{1-\rho}\sqrt{r} \quad \text{ weakly in }L^2(0,T;H^1(\Omega)),\\
\sqrt{1-\rho_\tau} \sqrt{b_\tau} &\rightharpoonup \sqrt{1-\rho}\sqrt{b} \quad \text{ weakly in }L^2(0,T;H^1(\Omega)).
\end{aligned}
\end{align}
The $L^{\infty}$-bounds, \eqref{equ12} and the fact that
\[\nabla \sqrt{1-\rho_\tau}\rightharpoonup \nabla \sqrt{1-\rho} \quad \text{ weakly in } L^2(0,T;L^2(\Omega)),\]
imply 
\begin{align*}
(1-\rho_\tau)\nabla \sqrt{r_\tau}&=\sqrt{1-\rho_\tau}\nabla (\sqrt{1-\rho_\tau} \sqrt{r_\tau})-\sqrt{1-\rho_\tau}\sqrt{r_\tau}\nabla \sqrt{1-\rho_\tau} \\
(1-\rho_\tau)\nabla \sqrt{b_\tau}&=\sqrt{1-\rho_\tau}\nabla (\sqrt{1-\rho_\tau} \sqrt{b_\tau})-\sqrt{1-\rho_\tau}\sqrt{b_\tau}\nabla \sqrt{1-\rho_\tau} 
\end{align*}
converge weakly in $L^1$ to the corresponding limits. The $L^2$ bounds imply also weak convergence in $L^2$:
\begin{align*}
(1-\rho_\tau)\nabla \sqrt{r_\tau}&\rightharpoonup (1-\rho)\nabla \sqrt{r}\quad \text{ weakly in }L^2(0,T;L^2(\Omega)),\\
(1-\rho_\tau)\nabla \sqrt{b_\tau}&\rightharpoonup (1-\rho)\nabla \sqrt{b}\quad \text{ weakly in }L^2(0,T;L^2(\Omega)).
\end{align*}
As $1-\rho_\tau\geq (1-\rho_\tau)^2$, we can pass to the limit inferior $\tau\to 0$ in
\begin{align*}
\begin{aligned}
&\quad\int_{\Omega} E(r_k,b_k)\,dx\,dy+C_0\tau\sum_{j=1}^k\int_{\Omega} (1-\rho_\tau)^2|\nabla\sqrt{r_\tau}|^2\,dx\,dy\\
&+C_0\tau\sum_{j=1}^k\int_{\Omega}  (1-\rho_\tau)^2|\nabla\sqrt{b_\tau}|^2+|\nabla\sqrt{1-\rho_\tau}|^2+|\nabla \rho_\tau|^2\,dx\,dy\\
&+\tau^2\sum_{j=1}^k R\left(\begin{pmatrix}
u_j\\ v_j
\end{pmatrix},\begin{pmatrix}
u_j\\v_j
\end{pmatrix}\right) \leq\int_{\Omega} E(r_0,b_0)\,dx\,dy+T C,
\end{aligned}
\end{align*}
attaining the entropy inequality \eqref{theorem1_2}.\qquad \endproof
\subsection{Existence for the general model}
The previous analysis can easily be extended to the case $0<\alpha\leq\frac{1}{2}$ and $\gamma_1-\gamma_2=\mathcal{O}(h)$ in \eqref{system}, i.e. leading to the system \eqref{equ.system2}.
The particular choice of parameters allows to obtain the following result. 
\medskip
\begin{theorem}{(Global existence)}\label{theorem2}
Let $0<\alpha\leq\frac{1}{2}$, $\gamma_1-\gamma_2=\mathcal{O}(h)$, $T>0$ and $(r_0,b_0):\Omega \to \mathcal{M}$, where $\mathcal{M}$ is defined by \eqref{equ:set}, be a measurable function such that $E(r_0,b_0)\in L^1(\Omega)$. Then there exists a weak solution $(r,b):\Omega\times (0,T)\to \overline{\mathcal{M}}$ to system \eqref{equ.system2} with periodic boundary conditions in $x$-direction and no-flux boundary conditions in $y$-direction satisfying the same regularity results and entropy dissipation inequality as stated in Theorem \ref{theorem1}.
\end{theorem}
\smallskip
The parameter regime $\gamma_1-\gamma_2=\mathcal{O}(h)$  and $0<\alpha \leq\frac{1}{2}$ results in additional terms in the entropy dissipation which can be estimated using Young's inequality in a way that inequality \eqref{entropyinequality} holds for the same constant $C_0$. In the existence proof, the limit $\tau \to 0$ requires some additional compactness results for the new terms which we obtain (as in the proof of Theorem \ref{theorem1}) by a generalized version of the Aubin-Lions lemma (cf. Lemma 7 in \cite{zamponi2015analysis}).

\section{Numerical simulations} \label{numerical_simulations}
In this last section we illustrate the behavior of the model with numerical simulations in spatial dimension two. In particular we compare the solutions of the minimal  model \eqref{system_rb}, i.e. no cohesion and no preference for dodging to one side with those of the system \eqref{equ.system2}.
All simulations have been carried out using the COMSOL Multiphysics Package with quadratic finite elements. We consider the domain $\Omega=[0,1]\times[0,0.1]$ representing a corridor, where we use a mesh consisting of 608 triangular elements and a BDF method with maximum time step $0.1$ to solve the corresponding system.  
We start with small perturbations of trivial stationary states to study if the system returns to this trivial solutions or results in a more complex one.
\subsection{Example I: Equilibration}
We consider system \eqref{system} with no adhesion and no preference to step to the right or left, i.e the minimal model \eqref{system_rb}.
We choose the parameters $\gamma_0=0.1$, $\gamma=0.2$ and $h=0.3$ and the initial values 
\begin{align}\label{initialvalues}
\begin{aligned}
r_0(x,y)&=c_r+0.02\sin(\pi x)\cos\left(\frac{\pi y}{0.1}\right),\\
b_0(x,y)&=c_b-0.02\sin(\pi x)\cos\left(\frac{\pi y}{0.1}\right),
\end{aligned}
\end{align}
with $c_r=c_b=0.4$. Figure \ref{f:ex1} illustrates the initial value $r_0$ and the solution $r_T$ to system \eqref{system_rb} at time $T=5$, where it can be seen that in this setting the solution returns back to the equilibrium state quickly. 

\begin{figure}[h!]
\begin{center}
\subfigure[$r_0$]{\includegraphics[height=50mm, width=60mm]{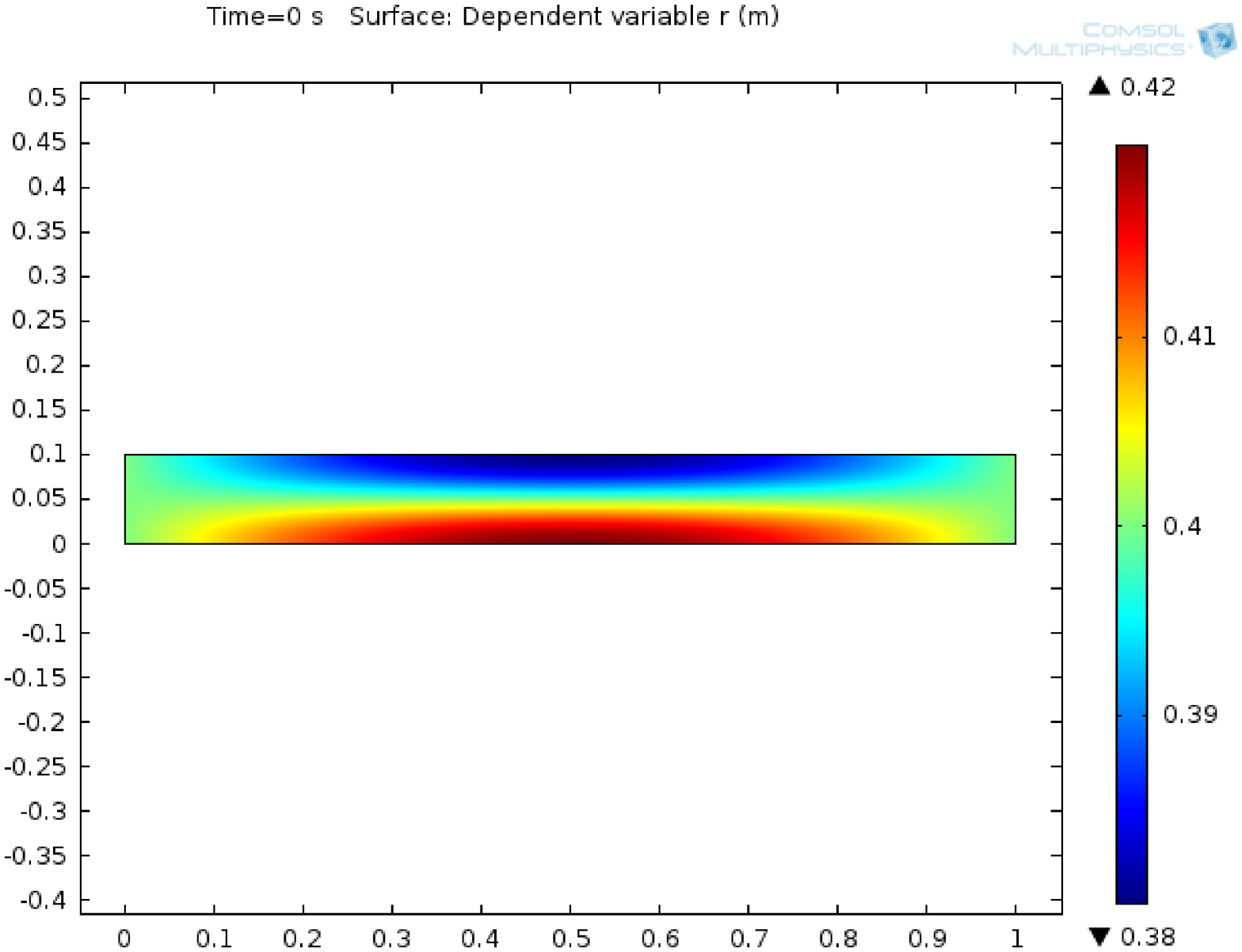}}
\subfigure[$r_T$ for $T=5$]{\includegraphics[height=50mm, width=60mm]{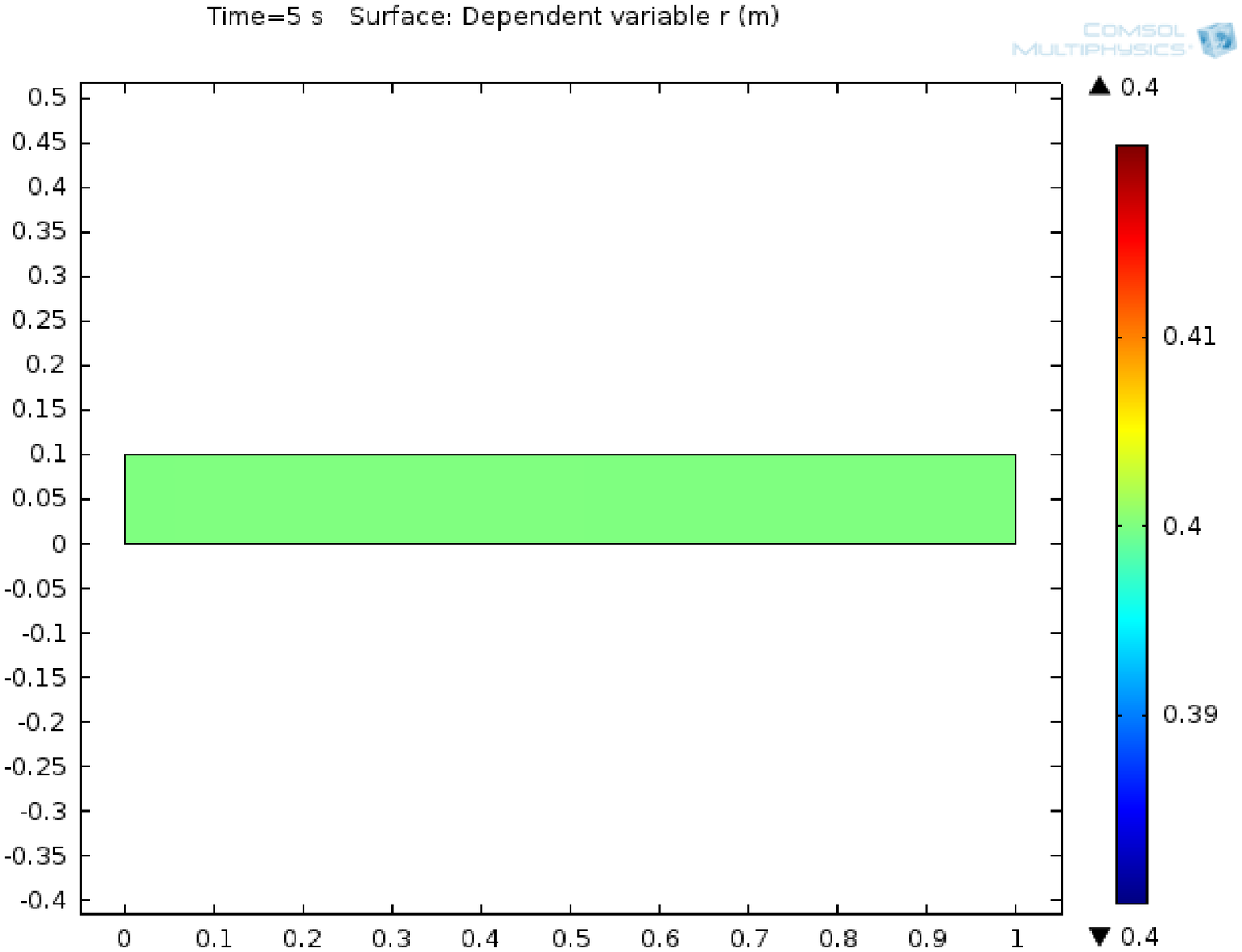}}
\caption{Example I: Red particle density returning to the constant stationary state after initial perturbation.}\label{f:ex1}
\end{center}
\end{figure}

\subsection{Example II: Lane-formation}\label{ex2}
The behavior of solutions to system \eqref{equ.system2} corresponding to the scaling $\gamma_1-\gamma_2=\mathcal{O}(h)$ is different. 
Setting $\gamma_0=0.001$, $\gamma_1=0.5$, $\gamma_2=0.4$, $\alpha=0.2$ and $h=0.1$, such that $\gamma_1-\gamma_2=\mathcal{O}(h)$, and choosing the same initial values \eqref{initialvalues} as above, we obtain weak lane formation illustrated in Figure \ref{f:ex2}. As $\gamma_1>\gamma_2$, the individuals have a tendency to step to the right. Therefore red individuals are highly concentrated on the bottom of the domain, whereas the blue individuals move to the top.

\begin{figure}[h!]
\begin{center}
\subfigure[$r_T$ for $T=5$]{\includegraphics[height=50mm, width=60mm]{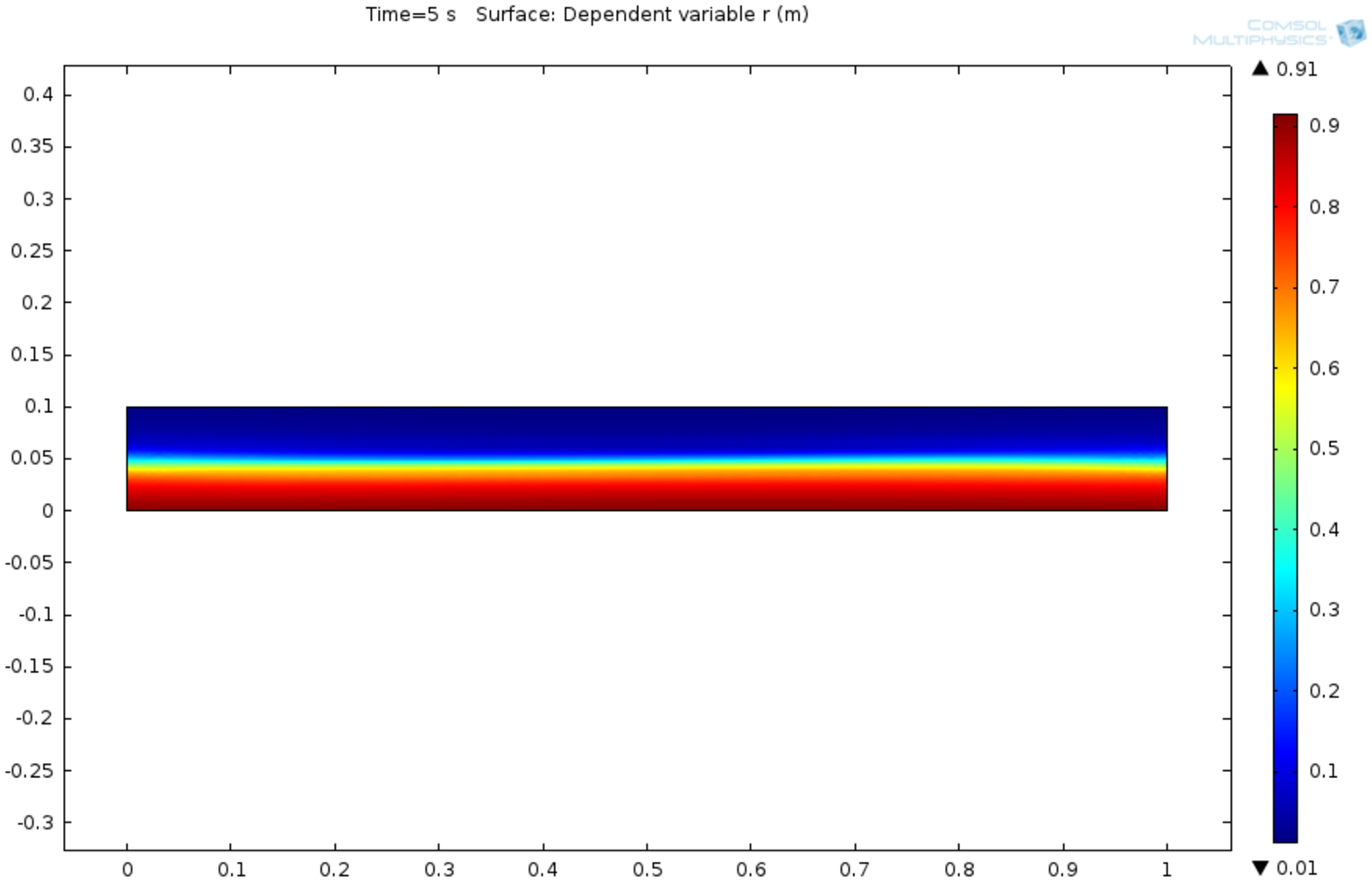}}
\subfigure[$b_T$ for $T=5$]{\includegraphics[height=50mm, width=60mm]{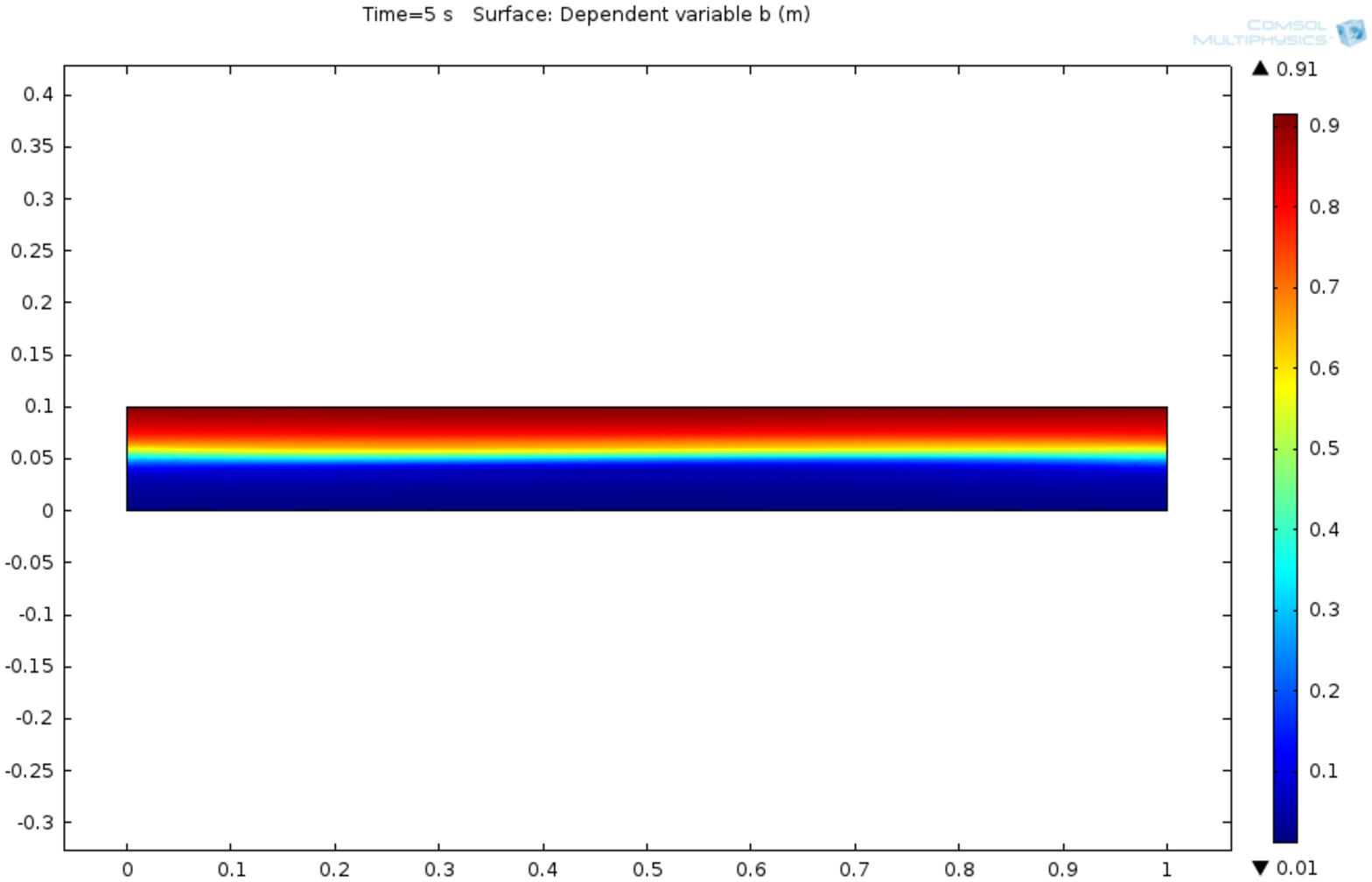}}
\caption{Example II: Red and blue particle distribution at time $T=5$ forming weak lanes.}\label{f:ex2}
\end{center}
\end{figure}
Figure \ref{f:plot1} shows the cross section of the two-dimensional solution $r_T$ at time $T=100$ for different initial masses. The parameters are the same as above while the constants $c_r=c_b=c$ in the initial values \eqref{initialvalues} are being varied. We observe that weak lane formation is more pronounced for smaller values of $\gamma_0$ as well as higher densities. The transition region around
$y = 0.05$ decreases for smaller diffusitivity and greater mass, while the behavior of $r$ is the same in the low density region $(0.05, 0.1)$ for all parameter sets.\\
\begin{figure}[h]
\begin{center}
\subfigure[$r$ at time $T=100$]{\includegraphics[scale=0.14]{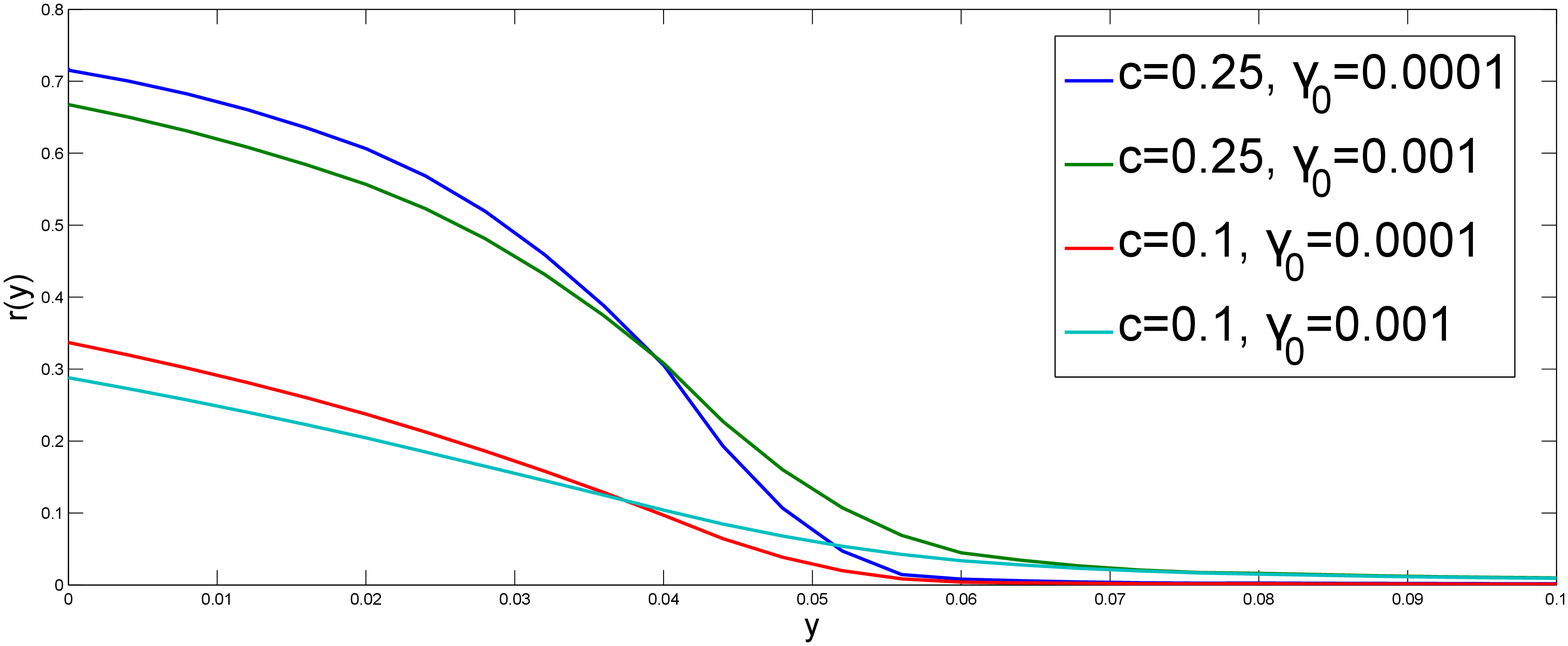} }
\subfigure[$b$ at time $T=100$]{\includegraphics[scale=0.14]{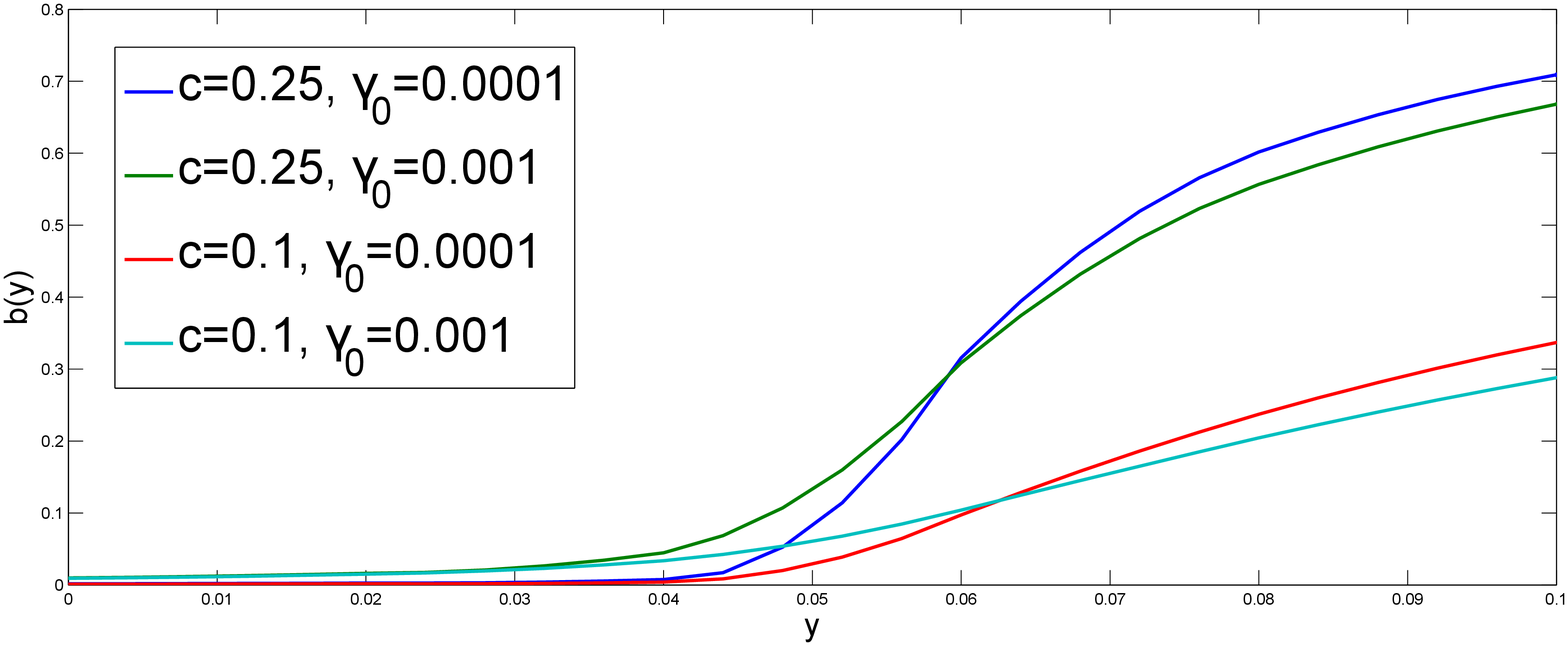}}
\caption{Example II: Red and blue particle density in $y$-direction at time  $T=100$.}\label{f:plot1}
\end{center}
\end{figure}
In the case of different masses $M_r$ and $M_b$ we observe asymmetric weak lane formation.
We choose initial values of the form \eqref{initialvalues}, i.e. $c_r=0.4$ and $c_b=0.1$. Figure \ref{f:plot2} shows the formation of such weak lanes due to the side-stepping mechanism even though the mass $M_b$ is smaller than $M_r$.
\begin{figure}[h]
\begin{center}
\includegraphics[scale=0.14]{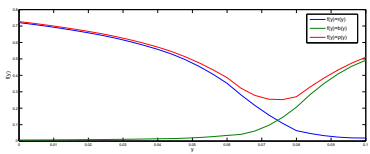}
\caption{Example II: Red and blue particle density as well as their sum $\rho$ at time $T=100$ in case of non-equal initial mass.}\label{f:plot2}
\end{center}
\end{figure}

\subsection{Example III: Jam} 
We conclude with a numerical simulation showing another well known phenomena in crowd dynamics, namely traffic jams or so-called 'freezing'. 
If the diffusion coefficients are small, i.e. $h=0.1$ and $\gamma_0=0.0001$, it may happen that the individuals cannot move in their walking direction 
any more as the initial masses are high compared to the diffusion coefficients. This ends in a jam or 'frozen' configuration as Figure \ref{f:ex1_1} illustrates.
\begin{figure}[h!]
\begin{center}
\subfigure[$r_T$ for $T=100$]{\includegraphics[height=50mm, width=60mm]{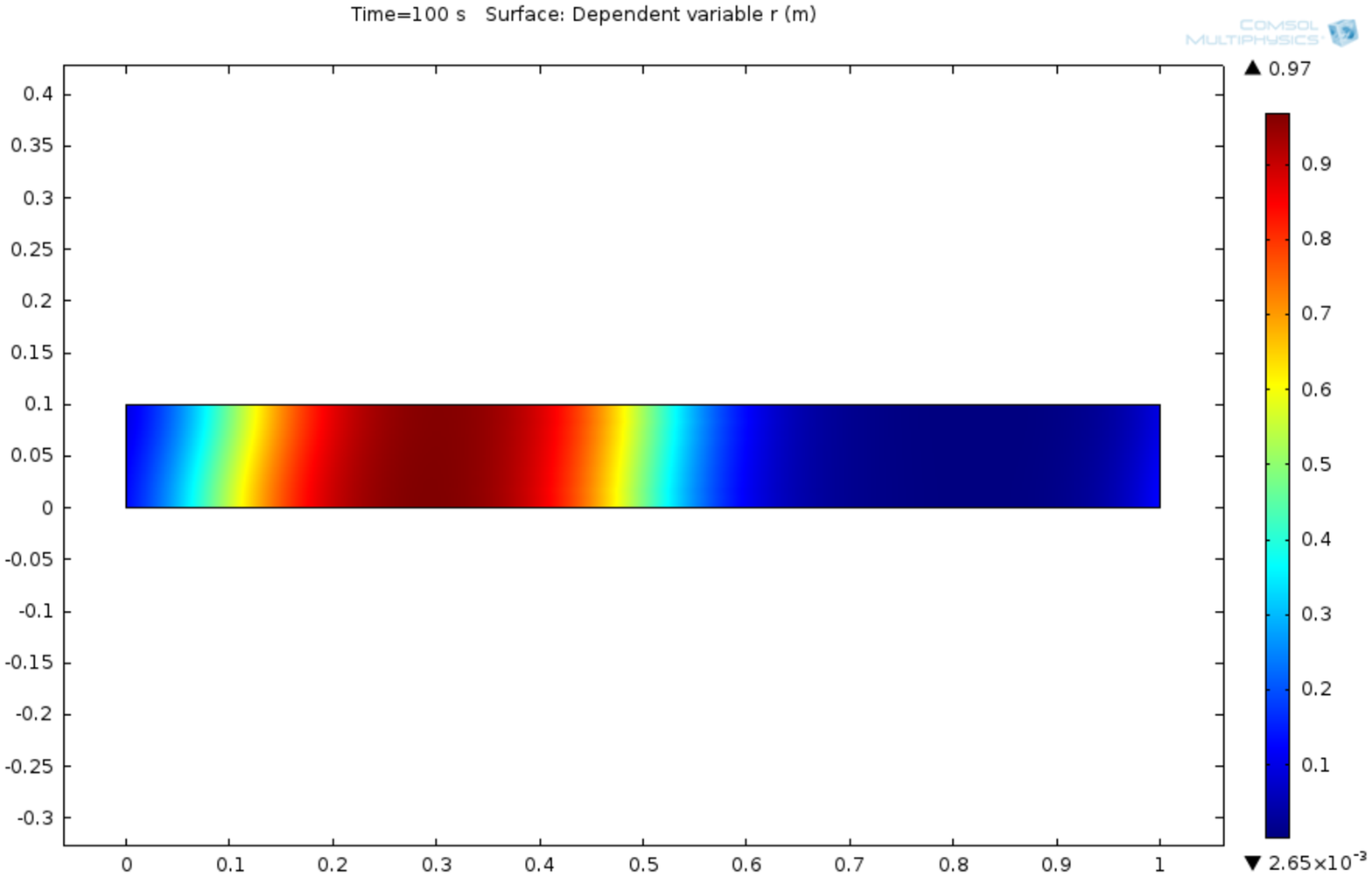}}
\subfigure[$b_T$ for $T=100$]{\includegraphics[height=50mm, width=60mm]{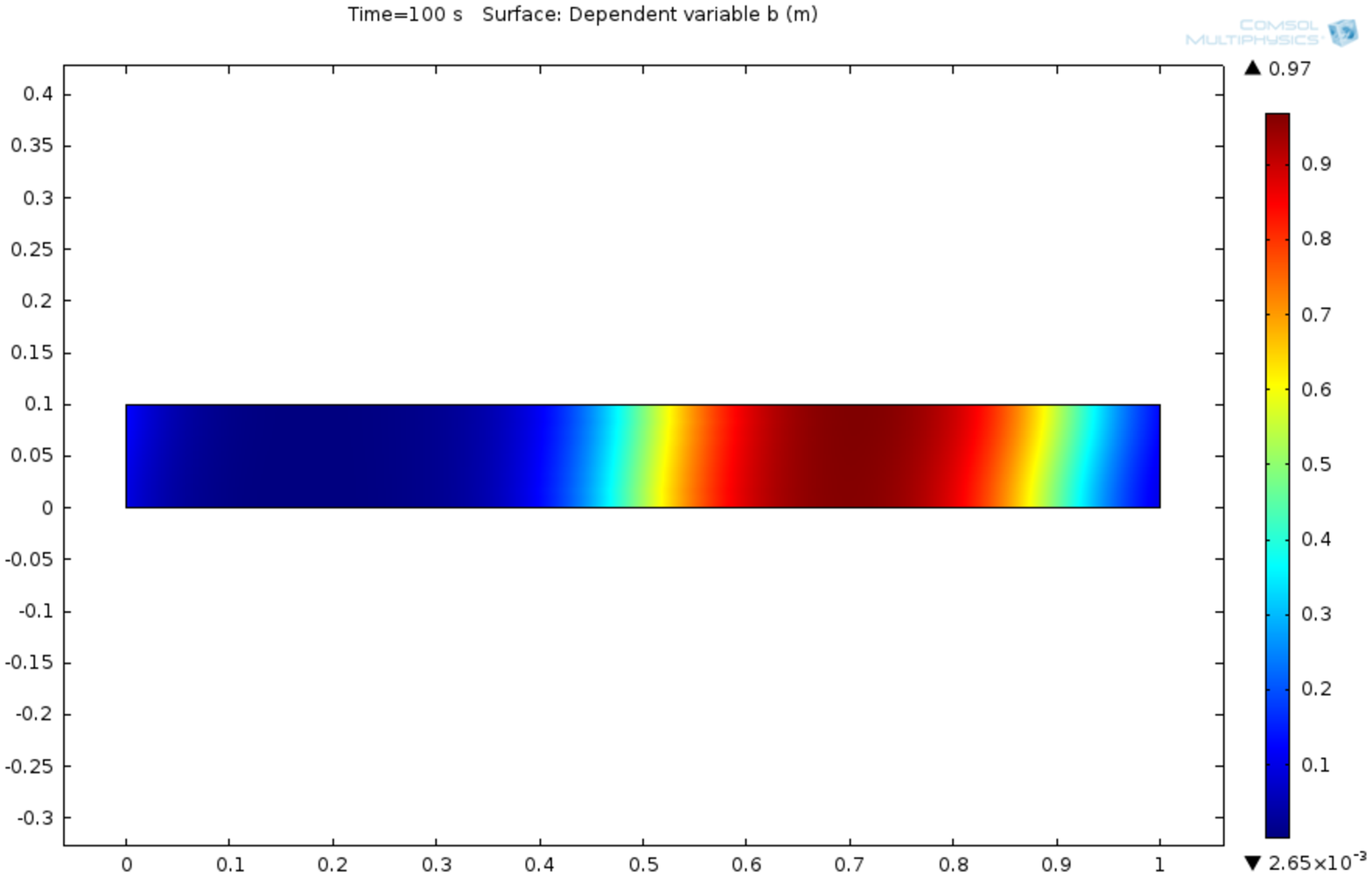}}
\caption{Example III: Congestion in the red and blue particle density resulting in a deadlock.}\label{f:ex1_1}
\end{center}
\end{figure}

\newpage

\section*{Acknowledgements}
The authors would like to thank Christian Schmeiser for his helpful suggestions and the discussions.


\bibliography{laneformation}
\bibliographystyle{plain}

\end{document}